\documentclass[12pt]{amsart}
\usepackage[margin=1.0in]{geometry}
\usepackage{amsmath,amsthm,amssymb}
\usepackage{tikz}
\usepackage{bm}
\usepackage{color}
\definecolor{grn}{rgb}{0,0.4,0}
\definecolor{dgrn}{rgb}{0.0,0.3,0.0}
\definecolor{dpur}{rgb}{0.3,0.0,0.6}
\usepackage[mathscr]{euscript}
\usepackage[colorlinks=true, linkcolor=dpur]{hyperref}
\numberwithin{equation}{section}
\newtheorem{prop}{Proposition}
\newtheorem{lemma}[prop]{Lemma}

\newtheorem{thm}[prop]{Theorem}
\newtheorem{cor}[prop]{Corollary}

\numberwithin{prop}{section}
\theoremstyle{definition}
\newtheorem{defn}[prop]{Definition}

\newtheorem*{thm*}{Theorem}
\newcommand{\brs}[1]{\left| #1 \right|}
\newcommand{\brk}[1]{\left[ #1 \right]}
\newcommand{\prs}[1]{\left( #1 \right)}
\newcommand{\sqg}[1]{\left\{ #1 \right\}}
\newcommand{\ip}[1]{\left\langle #1 \right\rangle}
\newcommand{\del}{\partial}

\newcommand{\hook}{\mathbin{\hbox{\vrule height2.4pt width4.5pt depth-2pt
\vrule height5pt width0.4pt depth-2pt}}}

\newcommand{\gG}{\Gamma}
\renewcommand{\gg}{\gamma}

\newcommand{\gL}{\Lambda}
\newcommand{\gd}{\delta}

\newcommand{\gU}{\Upsilon}
\newcommand{\gk}{\kappa}
\newcommand{\gl}{\lambda}
\newcommand{\bk}{\bar{k}}
\newcommand{\bi}{\bar{i}}
\newcommand{\bj}{\bar{j}}

\newcommand{\gw}{\omega}

\newcommand{\ga}{\alpha}
\newcommand{\gb}{\beta}
\renewcommand{\ge}{\epsilon}
\newcommand{\N}{\nabla}

\newcommand{\la}{\lambda}

\newcommand{\ten}{\otimes}

\newcommand{\til}[1]{\widetilde{#1}}
\newcommand{\nm}[2]{\brs{\brs{ #1}}_{#2}}

\renewcommand{\bar}[1]{\overline{#1}}
\newcommand{\hsp}{\hspace{0.5cm}}
\newcommand{\bRn}{\mathbb{R}^n}
\newcommand{\bR}{\mathbb{R}}

\newcommand{\CC}{\mathcal C}
\newcommand{\PP}{\mathcal P}

\newcommand{\YM}{\mathcal{YM}}

\DeclareMathOperator{\divg}{div}

\DeclareMathOperator{\supp}{supp}
\DeclareMathOperator{\End}{End}

\DeclareMathOperator{\Ad}{Ad}
\DeclareMathOperator{\Euc}{Euc}
\DeclareMathOperator{\refc}{ref}
\begin{document}

\title[Singularity formation of the Yang-Mills Flow]{Singularity formation of
the Yang-Mills Flow}

\author{Casey Kelleher}
\email{\href{mailto:clkelleh@uci.edu}{clkelleh@uci.edu}}
\address{Rowland Hall\\
         University of California\\
         Irvine, CA 92617}

\author{Jeffrey Streets}
\email{\href{mailto:jstreets@uci.edu}{jstreets@uci.edu}}
\address{Rowland Hall\\
         University of California\\
         Irvine, CA 92617}

\date{February 1, 2016}

\begin{abstract} We study singularity structure of
Yang-Mills flow in dimensions $n \geq 4$.  First we obtain a description of the
singular set in terms of concentration for a localized entropy quantity, which
leads to an estimate of its Hausdorff dimension.  We develop a theory of tangent
measures for the flow, which leads to a stratification of the singular set.  By
a refined blowup analysis
we obtain Yang-Mills connections or solitons as blowup limits at any point in
the singular set.
\end{abstract}
\maketitle

\section{Introduction}

Given $(M^n, g)$ a compact Riemannian manifold and $E \to M$ a vector
bundle, a one parameter family of connections $\N_t$ on $E$ is a solution to
Yang-Mills flow if
\begin{align*}
\frac{ \del \N_t}{\del t} &=\ - D_{\N_t}^* F_{\N_t}.
\end{align*}
This is the negative gradient flow for the Yang-Mills energy, and is a natural
tool for investigating its variational structure.  Global existence and
convergence of the flow in dimensions $n=2,3$ was established in \cite{Rade}. 
Finite time singularities in dimension $n=4$ can only occur via energy
concentration, as established in \cite{Struwe}.  More recently this result has
been refined in \cite{Feehan, Waldron} to show concentration of the self-dual
and antiself-dual energies.  Preliminary investigations into Yang-Mills flow in
higher dimensions have been made in \cite{Gastel, Naito,Weinkove}.

In this paper we establish structure theorems on the singular set for Yang-Mills
flow in dimensions $n \geq 4$.  Our results are inspired generally by results on
harmonic map flow, specifically \cite{LW1, LW2, LW3}.  The first main
result is a weak compactness theorem for solutions to Yang-Mills flow which
includes a rough description of the singular set of a sequence of solutions.  A
similar result for harmonic map flow was established in \cite{LW1}. 
Moreover, a related result on the singularity formation at infinity for a global
solution of Yang-Mills flow was established in \cite{HongTian}.  We include a
rough statement here, see Theorem \ref{fullweakcompactthm} for the precise
statement.

\begin{thm} \label{thm:weakcompactness} Fix $n\geq 4$ and let $E \to (M^n,g)$ be
a vector
bundle over a closed Riemannian manifold. Weak $H^{1,2}$ limits of
sequences of smooth solutions to Yang-Mills flow are weak solutions to
Yang-Mills flow which are smooth outside of a closed set $\Sigma$ of locally
finite
$(n-2)$-dimensional parabolic Hausdorff measure.
\end{thm}

The first key ingredients of the proof are localized entropy monotonicities for
the
Yang-Mills flow, defined in \cite{HongTian}, together with a low-entropy
regularity theorem \cite{HongTian}.  Fairly general methods allow for the
existence of
the weak limit claimed in Theorem \ref{thm:weakcompactness}, and the entropy
monotonicities are the key to showing that the singular set is small enough to
ensure that the weak limit is a weak solution to Yang-Mills flow.  The arguments
are closely related to those appearing in
\cite{HongTian,LW1,Tian}.

The second main result is a stratification of the singular set.  This involves
investigating tangent measures associated to solutions of Yang-Mills flow.  In
particular we are able to establish the existence of a density for these
measures together with certain parabolic scaling invariance properties.  One
immediate
consequence is that we can apply the general results of \cite{White} to obtain a
stratification of the singular set.  
See \S \ref{sec:tmstrat} for the relevant definitions.
\begin{thm}  \label{dimredthm}
For $0 \leq k \leq n-2$ let
\begin{equation*}
\Sigma_k := \left\{ z_0 \in \Sigma \mid \dim \prs{ \Theta^0 \prs{\mu^*, \cdot}}
\leq
k, \forall \mu^* \in T_{z_0}(\mu) \right\}.
\end{equation*}
Then $\dim_{\mathcal{P}}\prs{\Sigma_k} \leq k$ and $\Sigma_0$ is countable.
\end{thm}

The third main theorem characterizes the failure of strong convergence in the
statement of Theorem \ref{thm:weakcompactness} in terms of the bubbling off of
Yang-Mills connections. 
Again, an analogous result for harmonic maps was established in \cite{LW1}.  The
proof requires significant further analysis on tangent measures, leading to the
existence of a refined blowup sequence which yields the Yang-Mills connection. 
We give a rough statement below, see
Theorem \ref{weaktostrongH12} for the precise statement.  

\begin{thm}\label{thm:weaktostrong} Fix $n\geq4$ and let $E \to
(M^n,g)$ be a vector bundle over a closed Riemannian manifold. A
sequence of
solutions to Yang-Mills flow converging weakly in $H^{1,2}$ either converges
strongly in $H^{1,2}$, and the $(n-2)$-dimensional parabolic Hausdorff measure
of
$\Sigma$ vanishes, or it admits a blowup limit which is a Yang-Mills connection
on $S^4$.
\end{thm}

A corollary of these theorems is the existence of a either Yang-Mills connection
or Yang-Mills soliton as a blowup limit of arbitrary finite time singularities. 
For type I singularities the existence of soliton blowup limits was established
in \cite{Weinkove}, following from the entropy monotonicity for Yang-Mills flow
demonstrated in \cite{Hamilton}.  The existence of soliton blowup limits for
arbitrary singularities of mean curvature flow was established in
\cite{Ilmanen}, relying on the structure theory associated with Brakke's weak
solutions.  A preliminary investigation into the entropy-stability of Yang-Mills
solitons
was undertaken in \cite{CZ} and \cite{KS}.  Those results now apply to studying
arbitrary
finite-time singularities of Yang-Mills flow, as all admit singularity models
which are either Yang-Mills connections or Yang-Mills solitons.

\begin{cor}\label{cor:weaktostrong} Fix $n\geq4$ and let $E \to (M^n,g)$ be a
vector
bundle over
a closed Riemannian manifold.  Let
$\N_t$ a smooth solution to Yang-Mills flow on $[0,T)$ such that $\limsup_{t \to
T}
\brs{F_{\N_t}}_{C^0} = \infty$.  There exist a sequence $\{(x_i,t_i,\gl_i)\}
\subset M \times [0,T)\times [0,\infty)$ such that the corresponding blowup
sequence converges modulo gauge transformations to
either
\begin{enumerate}
\item A Yang-Mills connection on $S^4$.
\item A Yang-Mills soliton.
\end{enumerate}
\end{cor}
\subsection*{Acknowledgements}

The first author gratefully thanks Osaka University mathematics department, in
particular Toshiki Mabuchi and Ryushi Goto where much of the preliminary work
was performed, for their warm hospitality.  The first author also sincerely thanks Gang Tian and all of Princeton University, where much of the intermediate and final work was conducted, for providing such a friendly and productive atmosphere. The first author was supported by an
NSF Graduate Research Fellowship DGE-1321846.  The second author was supported
by the NSF via DMS-1341836, DMS-1454854 and by the Alfred P. Sloan Foundation
through a Sloan Research Fellowship.

\section{Background}

We will begin with a discussion of notation and conventions that are used
throughout the paper. We will then provide general analytic
background as well as a review of Yang-Mills flow and its key properties.

\subsection{Notation and conventions}\label{ss:notconv}

Let $(E,h) \to (M,g)$ be a vector bundle over a closed Riemannian manifold. Let
$S(E)$ denote the smooth
sections of $E$. For each point $x \in M$ choose a local orthonormal basis of
$TM$ given by $\{ \del_i \}$ with dual basis $\{ e^i \}$ and a local basis for
$E$ given by $\{ \mu_{\ga} \}$ with dual basis $\{ (\mu^*)^{\ga} \}$ for the
dual $E^*$. Let
$\Lambda^p(M)$ denote the set of smooth $p$-forms over $M$ and set $\Lambda^p(E)
:= \Lambda^p(M) \otimes S(E)$. Next set $\End E := E \ten E^*$, where $E^*$
denotes the dual space of $E$ and take
\begin{equation*}
\Lambda^p( \Ad E) := \{ \gw \in \Lambda^p(\End E) \mid h_{\ga \gg}
\gw_{\gb}^{\gg} = -
h_{\gb \gg} \gw_{\ga}^{\gg} \}.
\end{equation*}
The set of all bundle metric compatible connections on $E$ will be denoted by
$\mathcal{A}_{E}(M)$.  Thus, given a chart
containing $p \in M$ the action of a
connection $\N$ on $E$ is captured by the \textit{coefficient
matrices} $\Gamma  = ( \gG_{i \ga}^{\gb} e^i \otimes \mu_{\gb} \otimes
\mu_{\ga}^* )$, where
\begin{align*}
\N \mu_{\gb} &= \gG_{i \gb}^{\gd} e^i \otimes \mu_{\gd}.
\end{align*}
When
sequences of one-parameter families of connections $\{\N_t^i \}$ are in play we
will at times drop the explicit dependence on $t$ and $i$ for notational
simplicity.

\subsection{Weak solutions of Yang-Mills flow}

We first recall here the definitions of Sobolev spaces relevant to discussing
convergence of connections.  Refer to (\cite{Struwe} \S 1.3) for further
information. Using this we give the definition of a weak solution to Yang-Mills
flow.
\begin{defn}\label{defn:SobAdE} Fix $\N_{\refc}$ a background connection on $E$.
 The space $H^{l,p}(\Lambda^i(\Ad E))$ is the
completion of the space of smooth sections of $\Lambda^i(\Ad E)$ with respect to
the norm
\begin{equation*}
\brs{\brs{ \Upsilon }}_{H^{l,p}(\Lambda^i(\Ad E))} := \prs{\sum_{k=0}^l
\brs{\brs{ \N^{(k)}_{\refc} \Upsilon }}_{L^p(\Lambda^i(\Ad E))}^p }^{1/p} <
\infty.
\end{equation*}
We will say that a connection $\N$ is of
\emph{Sobolev class $H^{l,p}$}, and write $\N \in H^{l,p}$, if $\N = \N_{\refc}
+ \Upsilon$ where $\Upsilon
\in H^{l,p}\prs{\Lambda^1(\Ad E)}$.
\end{defn}
Now, for a vector bundle $E \to (M,g)$ over a Riemannian manifold, recall that
the
\emph{Yang-Mills energy} of a smooth connection $\N$ on $E$ with curvature
$F_{\N}$ is
\begin{equation*}\label{eq:YMe}
\mathcal{YM}(\N) := \frac{1}{2} \int_{M} \brs{F_{\N}}^2 \, dV.
\end{equation*}
From this we can consider the corresponding negative gradient flow, which is
easily shown to be the 
\emph{Yang-Mills flow}:
\begin{gather*}
\begin{split}
\left. \frac{\del \N_t}{\del t} \right. =&\ - D_{\N_t}^* F_{\N_t}.
\end{split}
\end{gather*}
With these definitions in place we can now define the notion of a weak solution
to the flow.

\begin{defn}\label{def:YMfwksol} A one-parameter
family $\N_t =
\N_{0} + \gU_t$ is a \emph{weak solution of Yang-Mills flow} on $[0,T]$ if
\begin{align*}
\gU_t \in L^1([0,T]; L^2(\Lambda^1(\Ad E))), \qquad F_{\N_t} \in
L^{\infty}([0,T];
L^2(\Lambda^2(\Ad E))),
\end{align*}
and if for all $\ga_t \in C^{\infty}([0,T]; H_1^2(\Lambda^2(\Ad E)))$ which
vanish at
$t=0,
t=T$, one has
\begin{align} \label{weaksoln}
\int_0^T \int_M \ip{\gU_t, \frac{\del \ga_t}{\del t}} - \ip{F_{\N_t}, \N_t
\ga_t}
\, dV\, dt = 0.
\end{align}
\end{defn}

\subsection{Blowup constructions}\label{ss:blowupconst}

Here we will give a discussion of the construction of blowup limits in the
setting of Yang-Mills flow.    First we define the fundamental scaling law.
\begin{defn}\label{defn:rescaledconn} Fix $U \subset \mathbb R^n$ and consider
the restricted bundle $E \to U$. Suppose $\N_t$ is a smooth solution to
Yang-Mills flow over $U$ on $[0,T)$.  Fixing a basis for $E$, $\N_t$ is
described by local coefficient matrices
 $\gG_t$.  Given $z_0 = (x_0,t_0) \in U \times [0,T)$ and $\la \in \bR$ we
define a connection $\N^{\la,z_0}_t$ via
coefficient matrices
\begin{align}
  \gG^{\la,z_0}_{t}\prs{x} = \la  \gG_{\la^2 t + t_0} \prs{ \la x + x_0}.
 \end{align}
 Typically the basepoint $z_0$ will be suppressed notationally when understood.
\end{defn}

Now consider a sequence $\sqg{(x_i,t_i,\la_i)} \subset M \times \bR \times
[0,\infty)$ with $\la_i \to 0$. Assuming $M$ is
compact there exists a subsequence such that $\{x_i\} \to x_{\infty} \in M$. 
Moreover,
we can pick a chart around $x_{\infty}$ so that the tail of the
sequence $\{ x_i \}$ is contained within this chart, identified with $B_1
\subset \mathbb R^n$.
For sufficiently large $i$, define a connection $\N_t^i$ via coefficient
matrices
\begin{equation*}
\Gamma^i_t(x):= \gG^{\gl_i,z_i}_t(x).
\end{equation*}
We call $\sqg{\N_t^i}$ an \emph{$(x_i,t_i,\la_i)$-blowup sequence}.  Note the
corresponding curvatures are scaled in the following manner,
\begin{align}\label{eq:curvrescale}
F_{\N_t^i}(x) = \la_i^2 F_{\N_{\la_i^2 t + t_i}} \prs{\la_ix + x_i}.
\end{align}
Observe that the domain of $\N^i_t$ contains $B_{\la_i^{-1}}(x_i) \times [
\frac{-t_i}{\la_i^2},
\frac{T - t_i}{\la_i^2}]$, so that the limiting domain is $\mathbb R^n \times
(-\infty,0]$.  If the points are chosen as a maximal blowup sequence so that the
curvatures are bounded, then these blowup solutions converge to a smooth ancient
solution to Yang-Mills flow.  However, in our analysis though we will be
choosing very general sequences and taking weak limits.

\subsection{Parabolic Hausdorff measures}

For any $0 \leq k \leq n$ and any $\Omega \subset \mathbb{R}^n$, the
$k$-dimensional Hausdorff measure of $\Omega$ is defined by
\begin{align*}
\mathcal{H}^{k}(\Omega)
&= \lim_{\gd \to 0} \mathcal{H}_{\gd}^{k}(\Omega)= \liminf_{\gd \to 0}
\left\{ \sum_i r_i^{k} \mid \Omega \subset \bigcup_i B_{r_i}(z_i), z_i \in
\Omega,
r_i \leq \delta \right\}.
\end{align*}
This leads to the definition of \emph{Hausdorff dimension}, i.e.
\begin{align*}
\dim_{\mathcal H}(\Omega) =&\ \inf\ \{d \geq 0\ |\ \mathcal H^d(\Omega) = 0 \}.
\end{align*}
Next, we define the \emph{parabolic metric} $\varrho$ on $\bRn \times \bR$ given
by, for $\prs{x,t},\prs{y,s} \in \bRn \times \bR$,
\begin{equation*}
\varrho \prs{\prs{x,t},\prs{y,s}} := \max \left\{ \brs{x-y}, \sqrt{\brs{t-s}}
\right\}.
\end{equation*}
Using this metric we can obtain the notion of parabolic Hausdorff dimension by
using covers by balls with respect to this metric.  In particular, for any $0
\leq \ell \leq n+2$ and any $\Omega \subset \mathbb{R}^n
\times \mathbb{R}$, the $\ell$-dimensional parabolic Hausdorff measure of
$\Omega$ is given by
\begin{align*}
\mathcal{P}^{\ell}(\Omega)
&= \lim_{\gd \to 0} \mathcal{P}_{\gd}^{\ell}(\Omega)= \liminf_{\gd \to
0} \left\{ \sum_i r_i^{\ell} \mid \Omega \subset \bigcup_i P_{r_i}(z_i), z_i \in
\Omega, r_i \leq \delta \right\},
\end{align*}
where, for $z_0 = (x_0,t_0) \in \bRn \times \bR$,
\begin{equation*}
P_r(z_0) := \left\{ z = (x,t)  \in \bRn \times \bR \mid \brs{x-x_0} < r,
\brs{t-t_0} < r^2
\right\}.
\end{equation*}
Using this we can then define the \emph{parabolic Hausdorff dimension}
\begin{align*}
\dim_{\mathcal P}(\Omega) := \inf\ \{ d \geq 0\ |\ \mathcal P^d(\Omega) = 0 \}.
\end{align*}
\section{Monotonicity Formulas} \label{s:monotonform}

In this section we observe some energy and entropy monotonicity formulas for
solutions to Yang-Mills flow which are central to the analysis below.

\subsection{Energy monotonicity}

\begin{lemma}\label{lem:8.2.1LWYM} Let $\N_t$ be a solution to Yang-Mills flow
on $M \times [t_1,t_2]$.  For any $\phi \in
C^1_0(M,[0,\infty))$,
\begin{align*}
\frac{1}{4}\int_M \prs{ \brs{F_{\N_{t_1}}}^2 - \brs{F_{\N_{t_2}}}^2 } \phi^2 \,
dV 
&= \int_{t_1}^{t_2} \int_M \prs{ \brs{\frac{\del \N_t}{\del t}}^2  + \ip{\frac{2
\N_t
\phi}{\phi} \hook F_{\N_t}, \frac{\del \N_t}{\del t}} }\phi^2 \, dV \, dt.
\end{align*}
\begin{proof}
We differentiate and find that
\begin{align*}
\tfrac{d}{dt} \brk{\tfrac{1}{2}\int_M \brs{F}^2 \phi^2 \, dV }
&= \int_M \ip{ F,\tfrac{\del F}{\del t} } \phi^2 \, dV \\
&= \int_M \ip{ F, D \brk{ \tfrac{\del \N}{\del t} } } \phi^2 \, dV \\
&= 2 \int_M \ip{ F,\N \brk{ \tfrac{\del \N}{\del t} } } \phi^2 \, dV \\
&= 2 \int_M \ip{ D^* F - 2 \tfrac{\N \phi}{\phi} \hook F, 
\tfrac{\del \N}{\del t}  } \phi^2 \, dV \\
&= 2 \int_M \ip{ - \tfrac{\del \N}{\del t} - 2 \tfrac{\N \phi}{\phi} \hook F, 
\tfrac{\del \N}{\del t} } \phi^2 \, dV \\
&= -2 \int_M \prs{ \ip{ \tfrac{\del \N}{\del t}, 2 \tfrac{\N \phi}{\phi} \hook
F} +
\brs{\tfrac{\del \N}{\del t}}^2} \phi^2 \, dV.
\end{align*}
Integrating both sides over $\brk{t_1,t_2}$ yields the result.
\end{proof}
\end{lemma}

\subsection{Entropy setup and scaling laws} \label{entropysetup}

Let $(M,g)$ be a Riemannian manifold. Let $\iota_M > 0$ be a lower bound for the
injectivity radius of $M$. Note that if $\N_t$ is a smooth solution to
Yang-Mills flow on $M \times [0,T)$, we can restrict it to any coordinate
neighborhood $B_{\iota_M} \subset \bRn$ is the Euclidean ball in $\bRn$ centered
at the
origin. 
Now fix $z_0 := \prs{x_0,t_0} \in
\mathbb{R}^n \times [0,\infty)$, and define
\begin{equation*}
G_{z_0}\prs{x,t} = \frac{e^{-\frac{\brs{x-x_0}^2}{4\brs{t-t_0}}}}{\prs{4 \pi
\brs{t-t_0}}^{n/2}}.
\end{equation*} 
We need to move this function onto the manifold $M$, and so we must localize. 
For $x_0 \in M$ we let $\mathcal{B}_{x_0}$ denote the set of \emph{cutoff
functions}, that is, all $\phi \in
C_0^{\infty}\prs{B_{\iota_M}(x_0),[0,\infty)}$ such that
\begin{align*}
\phi \in \brk{0,1}, \qquad \phi \equiv& 1 \text{ on }
B_{\tfrac{\iota_M}{2}}(x_0), \qquad \supp \phi \subset B_{\iota_M}(x_0).
\end{align*}
In this sense, given $z_0 = (x_0,t_0) \in M \times \bR$, for $\phi \in
\mathcal{B}_{x_0}$ one may consider the globally defined function $\phi G_{z_0}:
M \times \mathbb R \to [0,\infty)$.  Lastly, given $z_0 =
(x_0,t_0) \in M \times \bR$ and $R \in (0,\infty)$, we define
\begin{align*}
S_R(t_0) &:= M \times \{ t_0 - R^2 \},\\
P_R(z_0) &:= B_{R}(x_0) \times \prs{\brk{t_0 - R^2, t_0} \cap
(0,\infty)},\\
T_R(t_0) &:= M \times \prs{ \brk{t_0 - 4 R^2, t_0 - R^2} \cap (0,\infty)}.
\end{align*}
\begin{defn} \label{def:PhiPsifunc} Assume $\N_t$ is a solution to
Yang-Mills flow on $M \times [0,T)$. For $z_0 = (x_0,t_0) \in
M \times [0,T)$, $\phi \in \mathcal{B}_{x_0}$, and $R \in [0,
\min\{\iota_M, \sqrt{t_0}/2\} ]$, let
\begin{align*}
\Phi_{z_0}( R; \N_t) &:= \frac{R^4}{2} \int_{S_{R}(t_0)} \brs{F_{\N_t}}^2\phi^2 
G_{z_0} \, dV,\\
\Psi_{z_0}(R; \N_t) &:= \frac{R^2}{2}\int_{T_{R}(t_0)} \brs{F_{\N_t}}^2
\phi^2 G_{z_0} \, dV \, dt.
\end{align*}
\end{defn}

Next we record a fundamental scaling law for the entropy functionals which is
utilized in deriving the monotonicity formulas under Yang-Mills flow.  These
monotonicity formulas are shown in (\cite{HongTian}), but we include some brief
disussion of some properties for convenience, and also because we utilize some
of the calculations in the sequel.  We restrict the lemma to flat space for
convenience.

\begin{lemma}  \label{lem:entropyscaling} Fix $\N_t$ a solution to
Yang-Mills flow on $\prs{\bRn, g_{\Euc}} \times [0,T)$.
For all $z_0 = (x_0,t_0) \in \mathbb{R}^n \times [0,T)$, and $\prs{0 < R \leq
\sqrt{t_0}/2}$, setting $\phi
\equiv 1$ in Definition \ref{def:PhiPsifunc} yields
\begin{align*}
 \Phi_{z_0}(R;\N_t) &=\ \Phi_{z_0}\prs{1; \N^R_t},\\
 \Psi_{z_0}(R;\N_t) &=\ \Psi_{z_0}\prs{1; \N^R_t},
\end{align*}
where here $\N^R_t$ is the rescaled connection as defined in Definition
\ref{defn:rescaledconn}.
 \begin{proof} Without loss of generality we may take $z_0 = 0$. For notational
convenience we suppress the subscripts on $\Phi$, $\Psi$, and $G$. We fix $R >
0$
and consider a change of coordinates
\begin{align*}
x = R y, \qquad t = R^2 s.
\end{align*}
Then, rescaling coordinates and recalling the rescaling of the curvature tensor
(\ref{eq:curvrescale}),
\begin{align*}
dx = R^n \, dy, \qquad dt = R^2 \, ds, \qquad G(x,t) = R^{-n} G(y,s), \qquad
F_{\N^R_{s}} (y) = R^{2} F_{\N_{R^2 s}} (R y).
\end{align*}
 It follows that
\begin{align*}
\Phi(R;\N_t) &=\ \frac{R^4}{2} \int_{S_1} \brs{F_{\N_{R^2 s}}(Ry)}^2 \phi(y)G(y,
s) \,
dy\\
&=\ \frac{1}{2} \int_{S_1} \brs{F_{\N^R_s}(y)}^2 \phi(y) G(y,s) \, dy\\
&=\ \Phi(1; \N^R_t).
\end{align*}
Similarly,
\begin{align*}
\Psi(R;\N_t) &=\ \frac{R^4}{2} \int_{T_1} \brs{F_{\N_{R^2 s}}(R y)}^2
\phi(y)G(y, s) \, dy
\, ds\\
&=\ \frac{1}{2} \int_{T_1} \brs{F_{\N^R_s}(y)}^2 \phi(y)G(y,s) \, dy\, ds\\
&=\ \Psi\prs{1 ; \N^R_t}.
\end{align*}
The result follows.
\end{proof}
\end{lemma}

\subsection{Entropy monotonicities}\label{ss:entropmon}

In this section we recall the monotonicity formulae for $\Phi$ and $\Psi$,
established in \cite{HongTian}.  Again we record the proof on $\mathbb R^n$ for
convenience and as we will use parts of argument in the sequel.

\begin{prop}\label{prop:energymonform} Let $\N_t$ to be a smooth solution to
Yang-Mills flow for $\prs{\bRn,
g_{\Euc}} \times [0,T)$.
For all $z_0 = (x_0,t_0) \in \mathbb{R}^n
\times [0,T)$, and $0
< \rho \leq r < \sqrt{t_0}/2$, setting $\phi
\equiv 1$ in Definition \ref{def:PhiPsifunc} yields
\begin{align*}
\Phi_{z_0}(\rho;\N_t) &\leq \Phi_{z_0}(r;\N_t)\\
\Psi_{z_0}(\rho;\N_t) &\leq \Psi_{z_0}(r;\N_t).
\end{align*}

\begin{proof} We begin with the monotonicity statement for $\Phi$.  
We will include a generic cutoff function for purposes of a later Lemma. We fix
$R > 0$ and consider a
change of coordinates as in Lemma
\ref{lem:entropyscaling}.
As described there, it follows that
\begin{align*}
\Phi(R;\N_t) &=\ \frac{R^4}{2} \int_{S_1} \brs{F_{\N_{R^2 s}}(Ry)}^2
\phi^2(Ry)G(y, s) \,
dy.
\end{align*}
A crucial point here is that we are \emph{not} rescaling the connection as well.
One now differentiates and rescales back to obtain 
\begin{align*}
 \frac{\del}{\del R} \brk{\Phi(R;\N_t)} &=\ \frac{4}{R} \Phi(R;\N_t) + \brk{
R^3 \int_{S_R} \ip{F_{\N_t}, x \hook \del
F_{\N_t}}  \phi^2 G \, dx}_{I_1} \\
\hsp & + \brk{ 2 R^3 \int_{S_R} \ip{F_{\N_t}, t \prs{\tfrac{\del F_{\N_t}}{\del
t}}} \phi^2 G \, dx}_{I_2} + \brk{R^3 \int_{S_R} \brs{F_{\N}}^2 \phi x \hook \N
\phi\, dx}.
\end{align*}
To address $I_1$, we recall some coordinate formulas.
\begin{align*}
\N_i F_{jk \ga}^{\gb}
&= \del_i F_{jk \ga}^{\gb} + \gG_{i \mu}^{\gb}F_{jk \ga}^{\gb} - F_{jk
\mu}^{\gb}\gG_{i \ga}^{\mu},\\
\N_i F_{jk \ga}^{\gb} &=- \prs{ \N_k F_{ij \ga}^{\gb} + \N_j F_{ki \ga}^{\gb}}.
\end{align*}
Combining these we conclude that
\begin{align*}
\del_i F_{jk \ga}^{\gb} = - \prs{ \N_k F_{ij \ga}^{\gb} + \N_j F_{ki \ga}^{\gb}}
- \gG_{i \mu}^{\gb}F_{jk \ga}^{\mu} + F_{jk \mu}^{\gb}\gG_{i \ga}^{\mu}.
\end{align*}
With this in mind we manipulate $I_1$,
\begin{align*}
I_1
&= R^3 \int_{S_R } x^i\prs{ \N_k F_{ij \ga}^{\gb} + \N_j F_{ki
\ga}^{\gb}}F_{jk \gb}^{\ga} \phi^2 G \, dx \\
& \hsp + R^3 \int_{S_R } x^i \gG_{i \mu}^{\gb}F_{jk \ga}^{\mu} F_{jk
\gb}^{\ga} \phi^2 G \, dx  - R \int_{S_R } x^i F_{jk \mu}^{\gb}\gG_{i
\ga}^{\mu}F_{jk
\gb}^{\ga} \phi^2 G \, dx \\
&= 2 R^3 \int_{S_R } x^i\prs{ \N_k F_{ij \ga}^{\gb}}F_{jk \gb}^{\ga} \phi^2 G
\, 
dx\\
&= - 2 R^3 \int_{S_R }  F_{ij \ga}^{\gb} \N_k \brk{x^iF_{jk \gb}^{\ga} \phi^2 G
} \,
dx \\
&=  2 R^3 \int_{S_R }  \left[ F_{ij \ga}^{\gb} F_{ij \gb}^{\ga} + F_{ij
\ga}^{\gb}x^i \prs{\N_k F_{kj \gb}^{\ga}}  - \tfrac{1}{2 t} x^i  F_{ij
\ga}^{\gb} x^k F_{kj \gb}^{\ga} \right] \phi^2 G \, dx\\
& \hsp  - 4 R^3 \int_{S_R }  F_{ij \ga}^{\gb} x^iF_{jk \gb}^{\ga} \prs{\N_k\phi}
\phi G \, dx \\
&= \ -\tfrac{4}{R} \Phi(R;\N) + R^3 \int_{S_R } \left[ \tfrac{1}{t} \brs{x \hook
F}^2 - 2 \ip{x \hook F, D^* F} \right]  \phi^2 G \, dx \\
& \hsp  - 4 R^3 \int_{S_R }  F_{ij \ga}^{\gb} x^iF_{jk \gb}^{\ga} \prs{\N_k\phi}
 \phi G \, dx.
\end{align*}
Also we have
\begin{align*}
I_2
&= 2 R^3\int_{S_R} t \ip{F, \tfrac{\del F}{\del t}} \phi^2 G \, dx\\
& = - 2 R^3 \int_{S_R} t \ip{F, DD^* F}  \phi^2 G \, dx\\
& =  4 R^3 \int_{S_R} t F_{ij \ga}^{\gb} \N_i (D^* F)_{j \gb}^{\ga} \phi^2 G \,
dx\\
& =  R^3 \int_{S_R} \left[ 4 t \brs{D^* F}^2 - 2 \ip{x \hook F, D^* F}
\right] \phi G  \, dx - 8 R^3 \int_{S_R} t F_{ij \ga}^{\gb} (D^* F)_{j
\gb}^{\ga} \prs{\N_i\phi} \phi G \, dx.
\end{align*}
Combining these calculations gives
\begin{align}
\begin{split}\label{eq:delRmonotonicity}
  \frac{\del}{\del R} \brk{\Phi(R;\N_t)} &=\ \brs{t} R^3 \int_{S_R}
\brs{\frac{x}{t}
\hook F_{\N_t} - 2 D^*_{\N_t} F_{\N_t}}^2 \phi^2 G \,dx \\
& \hsp + 4 R^3 \int_{S_R} \prs{x^k F_{kj\gb}^{\ga} - 2 t \prs{D^* F}_{j
\gb}^{\ga} }F_{ij \ga}^{\gb} \prs{\N_i \phi} \phi G \, dx\\
& \hsp + R^3 \int_{S_R} \brs{F_{\N}}^2 \phi x \hook \N \phi\, dx.
\end{split}
\end{align}
In particular, when $\phi \equiv 1$, we have monotonicity, which yields the
first claim.

Next we prove the monotonicity of $\Psi$, only considering the case where $\phi
\equiv 1$.  We fix $R > 0$ and use the coordinate
change as in Lemma \ref{lem:entropyscaling} once more, and it follows that
\begin{align*}
\Psi(R;\N_t) &=\ R^4 \int_{T_1} \brs{F_{\N_{R^2 s}}(Ry)}^2 \phi^2(y) G(y, s) \,
dy
\, ds.
\end{align*}
Once again, crucially, we are \emph{not} rescaling the connection.  One now
obtains
\begin{align*}
\frac{\del }{\del R}\brk{ \Psi(R ;\N_t) } &= \frac{4}{R} \Psi(R;\N_t) + \brk{ 2
R \int_{T_{R}} \ip{F_{\N_t}, x \hook \del F_{\N_t}} \phi^2 G \, dx \,
dt}_{I_1}\\
& \hsp + \brk{4 R \int_{T_R} \ip{F_{\N_t}, t \prs{\tfrac{ \del F_{\N_t} }{\del
t}}} \phi^2 G \, dx \, dt}_{I_2}.
\end{align*}
Nearly identical estimates for $I_1$ and $I_2$ as in the case of $\Phi$ above
yield
\begin{align*}
\frac{\del}{\del R} \brk{\Psi(R;\N_t)} &=\ 2 R \int_{T_R} \brs{t}
\brs{\frac{x}{t} \hook F_{\N_t} -
2 D^*_{\N_t} F_{\N_t}}^2 \phi^2 G \, dx \, dt.
\end{align*}
The result follows.
\end{proof}
\end{prop}

Next we state the general monotonicity formula for $\Phi$ and $\Psi$ on
arbitrary Riemannian manifolds.  The proof is similar to that of Proposition
\ref{prop:energymonform}, incorporating further estimates due to the presence of
the cutoff function.  We state here the result of (\cite{HongTian} Theorem 2),
which applies to Yang-Mills-Higgs flow, and we just restrict the result to
Yang-Mills flow.  We point out that a similar result was claimed
in \cite{CS}, but uses definitions of $\Phi$ and $\Psi$ with incorrect scaling. 
Note that the notation for $\Phi$ and $\Psi$ agrees with various other
literature, but is reversed from that chosen in \cite{HongTian}.  Moreover, we
state an improved statement which is clearly implicit in \cite{HongTian}, simply
including an extra term in the inequality which is dropped in the statement in
\cite{HongTian}.

\begin{thm}[\cite{HongTian} Theorem 2, pp.448] \label{thm:8.2.2LWYM} Let $\N_t$
be a smooth
solution to Yang-Mills flow on $M \times [0,T)$. Then for $z_0 = (x_0,t_0) \in M
\times [0,T]$ and
$0 < R_1 \leq R_2 \leq \min \{\iota_M, \sqrt{t_0}/2
\}$, we have
\begin{align}
\begin{split}\label{Psimono}
\Psi_{z_0}(R_1;\N_t) & + \int_{R_1}^{R_2} r \int_{T_r(t_0)} \brs{t - t_0}
\brs{\frac{x - x_0}{2 \brs{t - t_0}} \hook F_{\N_t} - D^*_{\N_t} F_{\N_t}}^2
\phi^2 G_{z_0} \, dV \, dt  \, dr
\end{split}
\end{align}
\begin{align} \label{Phimono}
\begin{split}
& \leq e^{C(R_2-R_1)} \Psi_{z_0}(R_2;\N_t) + C (R_2 - R_1)\mathcal{YM}(\N_0),\\
\Phi_{z_0}(R_1;\N_t) & + \int_{R_1}^{R_2} r^3 \int_{S_r(t_0)} \brs{t - t_0}
\brs{\frac{x - x_0}{2 \brs{t - t_0}} \hook F_{\N_t} - D^*_{\N_t} F_{\N_t}}^2
\phi^2 G_{z_0} \, dV \, dr\\
& \leq e^{C(R_2-R_1)} \Phi_{z_0}(R_2;\N_t) + C \prs{R_2 - R_1}
\mathcal{YM}(\N_0).
\end{split}
\end{align}
\end{thm}

As the statement above makes clear, the functionals $\Phi$ and $\Psi$ are fixed
if
the connection satisfies a certain modified Yang-Mills type equation:

\begin{defn} Let $\N_t$
be a nontrivial smooth
one-parameter family of connections on
$\mathbb R^n \times (- \infty,0]$. Then $\N_t$ is a \emph{soliton} if
\begin{equation*}
D_{\N_t}^*F_{\N_t} = \frac{x }{2t} \hook F_{\N_t}.
\end{equation*}
\end{defn}

We end with a useful technical observation showing that the different entropies
$\Phi$ and $\Psi$ are uniformly equivalent, which exploits the monotonicity

\begin{lemma} \label{lem:cute}  Let $\N_t$ be a solution to Yang-Mills
flow on $M \times [0,T)$. There
exists a uniform constant $C$ such that for $z_0 = (x_0,t_0) \in M \times [0,T)$
and for $R$ with $0 < R \leq \min \{\iota_M, \sqrt{t_0}/2 \}$, we have
\begin{align*}
C^{-1} \Psi_{z_0}(R; \N_t) \leq \Phi_{z_0}(2 R;\N_t) \leq C \Psi_{z_0}(2
R;\N_t).
\end{align*}
\begin{proof} We give the proof on $\mathbb R^n$, in which case the monotonicity
does not involve the error term involving the Yang-Mills energy, with the
generalization to manifolds a straightforward extension.  Without loss of
generality we can consider the time interval to be $[-1,0]$ and choose $z_0 =
(0,0)$.    Then we have, using the monotonicity of $\Phi$ and a change of
variables,
\begin{align*}
\Phi(2R) 
&\geq\ \tfrac{1}{R} \int_{R}^{2R} \Phi(s) \, ds\\
&=\ \tfrac{R^3}{2} \int_{s=R}^{s=2R} \int_{M \times \{-s^2\}} \brs{F_s}^2 \phi^2
G \, dV \, ds\\
&=\ \tfrac{R^3}{2} \int_{t=-R^2}^{t=-4R^2} \tfrac{1}{2\sqrt{-t}}\int_{M \times
\{t\}} \brs{F_t}^2 \phi^2 G \, dV \, dt\\
&\geq\ c R^2 \int_{T_{R}(0)} \brs{F_t}^2 \phi^2G \, dV \, dt\\
&=\ c \Psi (R).
\end{align*}
Analogously we have
\begin{align*}
\Phi(R)
&\leq\ \tfrac{1}{R} \int_{R}^{2R} \Phi(s) \, ds\\
&=\ \tfrac{R^3}{4} \int_{s=R}^{s=2R} \int_{M \times \{-s^2\}} \brs{F_s}^2 \phi^2
G \, dV \, ds\\
&=\ \tfrac{R^3}{4} \int_{t=-R^2}^{t=-4R^2} \tfrac{1}{2\sqrt{-t}}\int_{M \times
\{t\}} \brs{F_t}^2 \phi^2 G \, dV \, dt\\
&\leq\ C R^2 \int_{T_{R}(0)} \brs{F_t}^2 \phi^2 G \, dV \, dt\\
&=\ C \Psi (R).
\end{align*}
The result follows.
\end{proof}
\end{lemma}

\subsection{Epsilon-regularity}

A central phenomenon in understanding the singularity formation of geometric
flows is that of $\ge$-regularity.  A result of this kind for Yang-Mills flow is
shown in \cite{HongTian}, relying centrally on the monotonicity formula for
$\Psi$ and the evolution equation for the curvature.  Once again we only state
the result for solutions to Yang-Mills flow though the result is shown for
Yang-Mills-Higgs flow in \cite{HongTian}.  We also point out that a similar
result is claimed in \cite{CS}, although it relies on the incorrectly defined
$\Psi$ functional.

\begin{thm}[\cite{HongTian} Theorem 4, pp.454]\label{thm:eregularity}Suppose
$\N_t$ is a solution to
Yang-Mills flow on $M \times [0,T)$.  There exist
constants $C,\gd,\ge_0 > 0$ depending on $(M, g)$ and $\YM(\N_0)$ so that given
$z_0 = (x_0,t_0) \in M \times [0,T)$ and $0 < R < \min \{ \iota_M, \sqrt{t_0}/2
\}$ such that
\begin{equation*}
\Psi_{z_0}(R;\N_t) < \ge_0,
\end{equation*}
one has
\begin{align*}
\sup_{P_{\gd R}(z_0)} \brs{F_{\N_t}}^2 \leq \frac{C}{ (\gd R)^{4}}.
\end{align*}
\end{thm}

\section{Weak Compactness and limit measures}\label{s:weakcompt}

In this section we establish a weak compactness result for solutions to
Yang-Mills flow satisfying certain weak convergence hypotheses.  In the first
subsection below we establish this theorem, and in the following subsection we
refine the analysis to show a number of properties of the limiting energy
densities and defect measures.

\subsection{Weak compactness theorem}

\begin{thm}\label{fullweakcompactthm} Suppose $\{\N^i_t\}$ is a sequence of
smooth solutions to Yang-Mills flow over $M \times [-1,0]$ with $\YM(\N^i_t)
\leq
\YM (\N^i_{-1}) <
C$. 
Moreover, suppose $\{\N_{-1}^i\} \to \N$ weakly in
$H^{1,2}_{loc}(\mathcal{A}_E(M))$,
and
\begin{itemize}
\item $\N^i_t \to \N_t$ in $L^2_{\mbox{loc}}(M \times [-1,0])$,
 \item $\frac{\del \N^i_t}{\del t} \to \frac{\del \N_t}{\del t}$ weakly in
$L^2_{\mbox{loc}}(M \times [-1,0])$,
 \item $F_{\N^i_t} \to F_{\N_t}$ weakly in $L^2_{\mbox{loc}}(M \times [-1,0])$.
\end{itemize}
Then $\N_t$ is gauge equivalent to a weak solution to Yang-Mills flow, and there
exists a closed set $\Sigma$ of locally finite $\prs{n-2}$-dimensional parabolic
Hausdorff measure such that $\N_t$ is a smooth solution on $(M \times (-1,0))
\backslash \Sigma$.

\begin{proof}
Set
\begin{equation*}
\Phi^i_{z_0}(r) := 
\begin{cases}
\Phi_{z_0}\prs{r;\N_t^i} & r \in \prs{0, \sqrt{1+t_0}}\\
\Phi_{z_0}\prs{\sqrt{1+t_0};\N_t^i}&  \text{ otherwise.}
\end{cases}
\end{equation*}
Now define the concentration set
\begin{equation*}
\Sigma := \bigcap_{r>0}\sqg{z \in M \times \brk{-1,0} \mid \liminf_{k\to\infty}
\Phi_z^k\prs{r} \geq \ge_0},
\end{equation*}
where $\ge_0$ is the constant of Theorem \ref{thm:eregularity}. To address the
theorem, we divide the proof up into three pieces: Lemma \ref{lem:Sigmacl},
Lemma \ref{weakcompactlemma3}, and Lemma \ref{weakcompactlemma10}.
\begin{lemma}\label{lem:Sigmacl}  $\Sigma$ is closed.
\begin{proof}
Let $\bar{z}$ lie in the closure of $\Sigma$ and $\{z_k\}_{k \in \mathbb{N}} \in
\Sigma$ with $z_k \to \bar{z}$. By the definition of $\Sigma$,
\begin{equation*}
\liminf_{k \to \infty} \liminf_{i \to \infty} \Phi^i_{z_k}(r) =\liminf_{k \to
\infty} \liminf_{i \to \infty}\brk{\frac{r^4}{2} \int_{\bRn \times \{ t_k -r^2
\}}  \brs{ F^i_t}^2 \phi^2G_{z_k} \, dV}\geq \ge_0. \end{equation*}
Note that $G_{z_k} \to G_{\bar{z}}$ on any closed sets not containing $\bar{z}$.
 Moreover, for fixed $i$ the function $\brs{F^i_t}^2$ is in $L^1$.  Therefore we
can fix $r > 0$, apply the dominated convergence theorem and interchange
$\liminf$ ordering by an elementary argument to conclude
\begin{align*}
\liminf_{i \to \infty} \frac{r^4}{2}\int_{M \times \sqg{ \bar{t} - r^2 }}  \brs{
F^i_t}^2 \phi^2G_{\bar{z}} \, dV =&\ \liminf_{i \to \infty} \lim_{k \to \infty}
\frac{r^4}{2} \int_{M \times \{\bar{t} - r^2\}} \brs{F_t^i}^2 \phi^2 G_{z_k} \,
dV\\
=&\ \liminf_{i \to \infty} \liminf_{k \to \infty} \frac{r^4}{2} \int_{M \times
\{\bar{t} - r^2\}} \brs{F_t^i}^2 \phi^2 G_{z_k} \, dV\\
\geq&\ \ge_0.
\end{align*}
Therefore $\bar{z} \in \Sigma$, so we conclude $\Sigma$ is closed.
\end{proof}
\end{lemma}
\begin{lemma} \label{weakcompactlemma3} $\N_t$ is gauge equivalent to a smooth
solution to Yang-Mills flow on $\prs{M \times (-1,0]} \backslash \Sigma$.
\begin{proof} Given $z \in \prs{\mathbb R^n \times (-1,0]} \backslash \Sigma$,
by construction there exists $r_0 > 0$ such that
\begin{equation*}
\liminf_{k \to \infty} \Phi^k_z (r_0) \leq \ge_0.
\end{equation*}
Passing to a subsequence and applying Lemma \ref{lem:cute}, we obtain an $\ge_0$
upper bound for $\Psi$, and by Theorem
\ref{thm:eregularity}, we conclude that
\begin{align*}
\sup_{P_{\gd r_0}(z)} \brs{F^k_t}^2 \leq \frac{C}{(\gd r_0)^{4}},
\end{align*}
for some universal constant $\gd > 0$.
Applying (\cite{Weinkove}, Theorem 2.2) we conclude uniform estimates on all
derivatives of curvature on a parabolic ball of radius $\frac{\gd r_0}{2}$.

Using the Uhlenbeck gauge-fixing Theorem (\cite{Uhlenbeck} Theorem 1.3) and the
gauge-patching argument of (\cite{DonB} Corollary 4.4.8) we can obtain a Coloumb
gauge on $B_{\frac{\gd r_0}{4}}$.  Moreover, by applying elliptic regularity
estimates (\cite{DonB} Lemma 2.3.11) and the Sobolev inequality we obtain
uniform pointwise estimates for the connection in the Coloumb gauge on
$B_{\frac{\gd r_0}{8}}$.  By applying the Yang-Mills flow PDE directly to this
gauge-fixed connection and using the previous estimates on the derivatives of
curvature we obtain uniform pointwise estimates for the gauge fixed connections
on $P_{\frac{\gd r_0}{8}}$.  Thus for each point $z_0$ we have constructed a
radius $\frac{\gd r_0}{8}$ and a sequence of gauge transformations for which the
parabolic ball of that radius has uniform control along some subsequence of
gauge-fixed connections.

Fix a compact set $K$ such that $K \cap \Sigma = \varnothing$.  For each $z \in
K$ there exist arbitrarily large values of $k$ and parabolic balls centered at
$z$ of the type described above.  This collection of parabolic balls covers $K$,
and since $K$ is compact we can choose a finite subcover, and also pass to a
subsequence of connections all of which have the bounds described above.  A
further application of the gauge-patching result (\cite{DonB} Corollary 4.4.8)
allows us to conclude the existence of a single gauge transformation, which,
when applied to our sequence, yields a sequence of connections with uniform
$C^{l,\ga}$ bounds. By the Arzela-Ascoli Theorem we obtain a further subsequence
converging on $K$.
\end{proof}
\end{lemma}

\begin{lemma} \label{weakcompactlemma5}  $\Sigma$
has locally finite $(n-2)$-dimensional parabolic Hausdorff measure.
\begin{proof} Fix a compact set $K$, and some $r_0 > 0$.  By Vitali's covering
lemma there exists some $l \in \mathbb{N}$, $\{ z_k \}_{k =1}^l \subset K \cap
\Sigma$ and $\{ r_k \}_{k = 1}^l \subset \prs{0,r_0}$ so that the sets $\{
P_{r_k}\prs{z_k} \}_{k=1}^l$, are mutually disjoint and $K \cap \Sigma$ is
covered by $\{ P_{5r_k}\prs{z_k} \}_{k=1}^l$.  Let $\bar{z}_k := z_k +
\prs{0,r_k^2}$ and fix some $\delta > 0$ to be determined later.  

The proof requires two different estimates on $G$ on different domains.  First,
on $(M \times \brk{t_k - 4 \delta^2 r_k^2 , t_k - \delta^2 r_k}) \backslash
P_{r_k}\prs{z_k}$ one has
\begin{equation*}
G_{z_k} \leq \delta^{-n} e^{- 1 / \prs{4 \delta}^2} G_{\bar{z}_k}.
\end{equation*}
Also, for points in $B_{r_k}(x^k) \times [t_k - 4 \gd^2 r_k^2, t_k - \gd^2
r_k^2]$ one has
\begin{align*}
G_{z_k} \leq C_{\gd} r^{-n}.
\end{align*}
We will also employ the estimate of Lemma \ref{lem:cute}, in particular
\begin{align*}
\Phi_{z_0}(R;\N_t) \leq C \Psi_{z_0}(R;\N_t).
\end{align*}
Combining the observations above we obtain, for all $k$, $i$
\begin{align*}
\begin{split}
\ge_0
&\leq \Phi_{z_k}^i \prs{\gd r_k}\\
&\leq C \Psi_{z_k}\prs{\delta r_k;\N^i_t} \\
&= C \gd^2 r_k^2  \int_{t_k - 4 \delta^2 r_k^2}^{t_k - \delta^2 r_k^2} \int_{M
\backslash B_{r_k}(x_k)}\brs{F^{i}_t}^2 G_{z_k} dV dt + C \gd^2 r_k^2
\int_{t_k - 4 \delta^2 r_k^2}^{t_k - \delta^2 r_k^2}
\int_{B_{r_k}(x_k)}\brs{F^{i}_t}^2 G_{z_k} dV dt \\
&\leq C \tfrac{e^{-1/(4 \delta)^2}}{4 \delta^{n}} \gd^2 r_k^2 \int_{t_k - 4
\delta^2 r_k^2}^{t_k - \delta^2 r_k^2} \int_{M \backslash B_{
r_k}(x_k)}\brs{F^{i}_t}^2 G_{\bar{z}_k} dV dt + C_{\gd} r_k^{2-n} \int_{t_k
- 4 \delta^2 r_k^2}^{t_k - \delta^2 r_k^2} \int_{B_{r_k}(x_k)}\brs{F^{i}_t}^2
dV dt \\
&\leq \brk{ C \tfrac{e^{-1/(4 \delta)^2}}{4\delta^{n}} \gd^2 r_k^2 \int_{t_k - 4
\delta^2 r_k^2}^{t_k - \delta^2 r_k^2} \int_{M}\brs{F^{i}_t}^2 G_{\bar{z}_k}
dV dt}_{I_1} + \brk{C_\gd r_k^{2-n} \int_{P_{r_k}(z_k)} \brs{F^{i}_t}^2 dV
dt}_{I_2}.
\end{split}
\end{align*}
Observe that we can estimate $I_1$ using Theorem \ref{thm:8.2.2LWYM} via
\begin{align*}
\gd^2 r_k^2 \int_{t_k - 4 \gd^2 r_k^2}^{t_k - \gd^2 r_k^2} \int_M
\brs{F^{i}_t}^2 G_{\bar{z}_k} \,dV dt
& =\ \gd^2 r_k^2 \int_{t_k + r_k^2 -
4 r_k^2 (1 + \gd^2)}^{t_k + r_k^2 - r_k^2(1 + \gd^2)} \int_M \brs{F^{i}_t}^2
G_{\bar{z}_k} \, dV dt\\
&=\ \Psi_{\bar{z}_k} \prs{r_k \sqrt{1 + \gd^2}; \N^i_t}\\
&\leq\ C \Psi_{\bar{z}_k} \prs{r_0;\N_t^i} + C(\prs{\YM\prs{\N_{-1}}})\\
&\leq\ C\prs{\YM \prs{\N_{-1}}}.
\end{align*}
Hence, since $\lim_{\gd \to 0} \frac{e^{-1/(4\gd)^2}}{4\gd^n} = 0$, we can
choose $\delta > 0$ sufficiently small so that $I_1 \leq
\frac{\ge_0}{2}$, which then implies that $I_2 \geq \frac{\ge_0}{2}$, which by
elementary manipulations gives
\begin{equation*}
r_k^{n-2}\leq \frac{C}{\ge_0} \int_{P_{r_k}(z_k))} \brs{F^{i}_t}^2 \, dV \, dt.
\end{equation*}
Therefore we have
\begin{align*}
\mathcal{P}_{5 r_0}^{n-2} \prs{P_R \cap \Sigma}
& \leq \sum_{k=1}^l \prs{5 r_k}^{n-2} \\
& \leq C \sum_{k=1}^l \int_{P_{r_k}(z_k)} \brs{F^{i}_t}^2 dV dt
\\
& \leq C \YM\prs{\N_{-1}}.
\end{align*}
Sending $r_0 \to 0$ allows us to conclude that $\mathcal{P}^{n-2}(\Sigma \cap
K) < \infty$ for any compact set $K$.  The result follows.
\end{proof}
\end{lemma}

\begin{lemma} \label{weakcompactlemma10}  $\N_t$ is a
weak solution to Yang-Mills flow.
\begin{proof}  We verify
(\ref{weaksoln}) by approximating via cutoff functions which excise the singular
set $\Sigma$.  To construct these functions, first consider the coverings
constructed in Lemma \ref{weakcompactlemma5}.  In particular, given any
$r_0>0$ there is some finite cover $\{ P_{r_i}(z_i)\}_{i=1}^{l}$ of
$\Sigma$, for some $l \in \mathbb{N}$ with $r_i < r_0$ satisfying
\begin{equation} \label{weakcompactlemma105}
\sum_{i=1}^{l} r_i^{-4} \brs{P_{r_i}(z_i)} \approx \mathcal
P_{5r_0}^{n-2}(K
\cap \Sigma) \leq C \YM \prs{\N_{-1}},
\end{equation}
where here $\brs{\cdot}$ denotes the Lebesgue measure on $\bRn \times
\mathbb{R}$.

Let $\phi \in C_0^{\infty}(P_2, [0,\infty))$ be a standard bump function
satisfying
$0 \leq \phi \leq 1$ and $\phi \equiv 1$ on $P_1$. For all $i \in
\mathbb{N}$, define
\begin{equation*}
\phi_i (x,t) := \phi \prs{\tfrac{x-x_i}{r_i},\tfrac{t-t_i}{r_i^2}}.
\end{equation*}
Let $\ga \in C^{\infty}([0,T]; L^2(\Lambda^2(\Ad E)))$ and arbitrary and set
\begin{align*}
\eta
:=  \alpha \inf_i \prs{1-\phi_i} \in C_0^{\infty}\prs{ \prs{\bRn \times
(-1,0)} \backslash \Sigma}.
\end{align*}
Note that by definition, $\eta \to \alpha$
almost
everywhere as $r_0 \to 0$.  Furthermore, observing that $\supp \eta \subset
\prs{\bRn \times \prs{-1,0}} \backslash \Sigma$, it follows from Lemma
\ref{weakcompactlemma3} that, setting $\gU_t = \N_{\refc} - \N_t$, we have
\begin{align*}
\int_{-1}^0 \int_{M} \ip{\gU, \tfrac{\del \eta}{\del t}} - \ip{F, D \eta} \, dV
\, dt =
0.
\end{align*}
Using this we can estimate
\begin{align*}
\int_{-1}^0 & \int_M \ip{ \gU, \tfrac{\del \alpha}{\del t}} - \ip{F, D \alpha}
dV
\, dt\\
&=\ \brs{ \int_{-1}^0 \int_M \ip{ \gU, \tfrac{\del (\ga - \eta)}{\del t}} -
\ip{F,
D(\ga - \eta)} \, dV \, dt }\\
&=\ \brs{\int_{-1}^0 \int_M \ip{ \tfrac{\del \gU}{\del t}, \ga - \eta} - \ip{F,
[1
- \inf_i (1 - \phi_i)] D \ga} - \ip{F, \ga \wedge d( \inf_i (1 - \phi_i))}
\, dV \, dt }\\
&=\ \brs{ I_1 + I_2 + I_3 } \\
& \leq \sum_{j=1}^3 \brs{I_j}.
\end{align*}
First, since we have almost everywhere convergence of $\ga$ to $\eta$ and
$\frac{\del \gU}{\del t}$ is in $L^2$ we have $\lim_{r_0 \to 0} I_1 = 0.$ 
Similarly since $[1 - \inf_i (1 - \phi_i)]$ goes to zero uniformly one has that
$\lim_{r_0 \to 0} I_2 = 0$.  For the final term, we observe using H\"older's
inequality and (\ref{weakcompactlemma105}) that
\begin{align*}
\lim_{r_0 \to 0} \brs{I_3} &\leq\ C \lim_{r_0 \to 0} \nm{F}{L^2( \cup_i
P_{r_i}(z_i))} \left[ \int_{-1}^0 \int_M \brs{\N \inf_{1 \leq i \leq l}
(1 - \phi_i)}^2 \, dV \, dt 
\right]^{\frac{1}{2}}\\
&\leq\ C \lim_{r_0 \to 0} \nm{F}{L^2( \cup_i P_{r_i}(z_i))} \left[
\sum_{i=1}^{l}
r_i^{-2} \brs{P_{r_i}(z_i)} \right]^{\frac{1}{2}}\\
&\leq\ C \lim_{r_0 \to 0} r_0 \left[ \sum_{i=1}^l r_i^{-4} \brs{P_{r_i}(z_i)}
\right]^{\frac{1}{2}}\\
&=\ 0.
\end{align*}
The lemma follows.
\end{proof}
\end{lemma}

Combining the result of Lemma \ref{lem:Sigmacl}, Lemma \ref{weakcompactlemma3},
and Lemma \ref{weakcompactlemma10}, the results of Theorem
\ref{fullweakcompactthm} follow.
\end{proof}
\end{thm}

\subsection{Structure of limit measures} \label{limmeasstruct}

Assume the setup of Theorem \ref{fullweakcompactthm}.  Observe that the measures
\begin{equation*}
\sqg{ \brs{F_{\N_t^i}}^2 dV \, dt } \text{ and }\sqg{\brs{\frac{\del
\N^i_t}{\del t}}^2
dV \, dt }
\end{equation*}
admit subsequences converging in the sense of Radon measures to some
limit measures.  We can compare these to the measures induced by the weak
$H_1^2$ limit $\N$ to define measures $\mu,$ $\nu$ and $\eta$ via
\begin{align*}
\brs{F_{\N^i_t}}^2 \, dV \, dt \to&\ \brs{F_{\N^{\infty}_t}}^2 \, dV \, dt + \nu
\equiv \mu,\\
\brs{\frac{ \del \N^i_t}{\del t}}^2 \, dV \, dt \to&\ \brs{ \frac{\del
\N^{\infty}_t}{\del t}}^2 \, dV \, dt + \eta.
\end{align*}

The remainder of the section consists of a series of lemmas further refining the
nature of these measures.

\begin{lemma}\label{lem:8.2.5LWYM} Fix $z = (x,t) \in M \times [-1 ,0]$
and $\phi \in \mathcal{B}_x$.  Then
\begin{equation*}
\Theta(\mu,z) := \lim_{R \to 0} R^2\int_{T_R(z)} \phi^2(x) G_z(x,t) \, d
\mu(x,t)
\end{equation*}
exists and is upper semicontinuous for all $z \in M \times [0, \infty)$.
 Moreover,
\begin{align*}
 \Sigma = \left\{ z \in M \times (0,\infty) \mid \ge_0 \leq \Theta(\mu,z) <
\infty \right\}.
\end{align*}

\begin{proof} We consider the limit as $i \to \infty$ in the monotonicity
inequality (\ref{Psimono}).  In particular, for $0 < R \leq R_0$, let
\begin{align*}
f(R,d\mu) =&\ e^{CR} \left[ \frac{R^2}{2} \int_{T_{R}} \phi^2 G_z \, d \mu + C
e^{CR} R \YM(\N_{-1}) \right].
\end{align*}
We observe that (\ref{Psimono}) implies that
\begin{align*}
f(R,\brs{F_{\N^i_t}}^2 dV) =&\ e^{C R} \left[ \Psi_{z_0}(R,\N^i_t) + C R
\YM(\N_{-1}) \right]\\
\leq&\ e^{CR} \left[ e^{C(R_0 - R)} \Psi_{z_0}(R,\N_t^i) + C(R_0 - R)
\YM(\N_{-1}) + C R \YM(\N_{-1}) \right]\\
=&\ f(R_0,\brs{F_{\N^i_t}}^2 dV).
\end{align*}
Using that $\brs{F_{\N_t^i}}^2 dV$ converges to $d\mu$, it follows that $f(R,d
\mu)$ is monotone nondecreasing as well.  It follows that $\lim_{R \to 0} f(R,d
\mu)$ exists, and by elementary arguments the limit defining $\Theta$ also
exists, and is upper semicontinuous.
\end{proof}
\end{lemma}

\begin{lemma} \label{weaktostronglem10} For $\PP^{n-2}$-almost everywhere $z \in
\Sigma$, one
has
\begin{align*}
 \lim_{R \to 0} R^{2-n} \int_{P_R(z)} \brs{F_{\N_t}}^2 \, dV \, dt = 0,
\qquad
\Theta(\mu,z) = \Theta(\nu,z) \geq \ge_0.
\end{align*}
\begin{proof} To show the first claim, let
\begin{align*}
 K_j = \sqg{ z \in \Sigma \ |\ \limsup_{R \to 0} R^{2-n} \int_{P_R(z)}
\brs{F_t}^2 \,
dV \, dt > j^{-1}
}.
\end{align*}
We will show that the $(n-2)$-parabolic Hausdorff measure of $K_j$ is zero for
each $j$, which suffices.  Fixing some $\gd > 0$ we can apply Vitali's covering
lemma to obtain a covering of $K_j$ by disjoint parabolic balls
$P_{r_{k}}(z_{k})$
with $z_{k} \in K_j, 5 r_{k} \leq \gd$, such that $K_j \subset \bigcup P_{5
r_{k}}(z_{k})$.  It follows that there
exists $C > 0$ such that
\begin{align*}
 \PP^{n-2}(K_j) &\leq\ \lim_{\gd \to 0} \sum_{k} (5 r_{k})^{n-2}\\
 &\leq\ C j \lim_{\gd \to 0} \int_{N_{\gd}(\Sigma)} \brs{F_t}^2 \, dV \, dt\\
 &=\ 0,
\end{align*}
where $N_{\gd}(\Sigma)$ indicates the parabolic $\gd$-tubular neighborhood of
$\Sigma$, and the last line follows by the dominated convergence theorem.  The
second claim now follows from the first and the definitions of $\mu,\nu$.
\end{proof}
\end{lemma}

\begin{lemma}  \label{wktostronglemma5} For
$\PP^{n-2}$-almost everywhere $z \in \Sigma$.
\begin{align*}
\lim_{r \to 0} \lim_{i \to \infty} r^{4-n} \int_{P_r(z)} \brs{ \frac{\del
\N^i_t}{\del t}}^2 \, dV \, dt = 0.
\end{align*}
\begin{proof} We will show that for any $\ge > 0$, the set
\begin{align*}
\CC_{\ge} := \sqg{z \in \Sigma \mid \liminf_{r \to 0} \liminf_{i \to \infty}
r^{4-n}
\int_{P_r(z)} \brs{\frac{ \del \N^i}{\del t}}^2 \, dV \, dt \geq \ge }
\end{align*}
satisfies $\PP^{n-4}(\CC_{\ge}) < \infty$.  Given this, we can express
\begin{align*}
\Sigma' := \left\{ z \in \Sigma \mid \liminf_{r \to 0} \liminf_{i \to \infty}
r^{4-n}
\int_{P_r(z)} \brs{\frac{\del \N^i}{\del t}}^2 \, dV \, dt = 0 \right\} =
\Sigma\
\backslash\ \prs{ \bigcup_{n \in \mathbb N} \CC_{2^{-n}}}.
\end{align*}
In particular, $\Sigma'$ can be obtained from $\Sigma$ by removing a countable
union of sets of finite $\PP^{n-4}$ measure, which has zero $\PP^{n-2}$ measure
by a
standard argument.

To show $\PP^{n-4}(\CC_{\ge}) < \infty$, fix a $\delta > 0$, and apply Vitali's
covering lemma to obtain a collection $\{ z_k \}_{i \in \mathbb{N}} \subset
\Sigma$ and $r_k \in (0,\delta)$ satisfying that $\{ P_{r_k}(z_k) \}$ are
mutually disjoint, $\{ P_{5 r_k}(z_k) \}$ cover $\Sigma$, and furthermore there
is some subsequence $\{ \N^{i}_t \}$ so that for all $k,i$,
\begin{equation*}
r_k^{4-n} \int_{P_{r_k}(z_k)} \brs{\tfrac{\del \N^{i}}{\del t}}^2 \, dV \, dt
\geq \ge.
\end{equation*}
Using this we obtain
\begin{align*}
\mathcal{P}^{n-4}_{5 \gd} \prs{\CC_{\ge}}
&\leq \sum_{k=1}^{\infty} \prs{5 r_k}^{n-4} \\
&= 5^{n-4} \sum_{k=1}^{\infty} r_k^{n-4} \\
&\leq \tfrac{5^{n-4}}{\ge} \sum_{k=1}^{\infty} \int_{P_{r_k}(z_k)} \brs{
\tfrac{\del \N^{i}}{\del t} }^2 \, dV \, dt \\
&\leq \tfrac{5^{n-4}}{\ge} \int_{\bigcup_{k=1}^{\infty} P_{r_k}(z_k)}
\brs{  \tfrac{\del \N^{i}}{\del t}}^2 \, dV \, dt \\
&\leq C(n,\ge) \int_0^2 \int_{B_2} \brs{ \tfrac{\del \N^{i}}{\del t}}^2 \, dV \,
dt\\
&\leq C(n,\ge, \mathcal{YM}(\N^i_{-1})),
\end{align*}
where the last line follows via the Yang-Mills energy monotonicity.  Sending
$\gd$ to zero
proves that $\PP^{n-4}(\CC_{\ge}) < \infty$, finishing the proof.
\end{proof}
\end{lemma}

\begin{lemma} \label{wktostronglem15} The density function $\Theta(\mu,x)$ is
$\PP^{n-2}$-approximately continuous at $\PP^{n-2}$-almost every $x \in \Sigma$.
 That is, for all $\PP^{n-2}$-a.e. $z \in \Sigma$ one has that for all $\ge >
0$,
\begin{align*}
 \lim_{r \to 0} r^{2-n} \PP^{n-2} \left( \{ w \in P_r(x) \cap \Sigma\ |\
\brs{\Theta(\mu,w) - \Theta(\mu,z)} > \ge \} \right) = 0.
\end{align*}
\begin{proof} Note that for a given $x \in \Sigma$, the density $\Theta(\mu,x)$
is upper semicontinuous, so the set
\begin{equation*}
A_c := \sqg{ z \mid \Theta (\mu, z) < c}
\end{equation*}
 is open. Therefore for any $c_1, c_2 \in [0,\infty)$ with $c_1 < c_2$, the set
$A_{c_2} \backslash A_{c_1}$ is a Borel set and thus measurable.  Hence
 \begin{equation*}
 E_i := \sqg{ z \in \Sigma \mid \frac{(i-1) \ge}{2} \leq \Theta(\mu,z) < \frac{i
\ge}{2} } = A_{ \frac{i \ge}{2}} \backslash A_{\frac{(i-1)\ge}{2}},
 \end{equation*}
 is a Borel set.  Note that, by the definition of $E_i$,
 \begin{equation*}
 \mathcal{P}^{n-2} \prs{\Sigma \backslash \bigcup_i E_i } = 0.
 \end{equation*}
 For all $x \in E_i$, by applying Theorem 3.5 of \cite{Simon} to the measure
$\PP^{n-2}$ we have that
 \begin{align*}
 \lim_{R \to 0} R^{2-n} \mathcal{P}^{n-2} &\prs{ \sqg{ y \in P_r(x) \cap
\Sigma \mid \brs{ \Theta(\mu, w) - \Theta(\mu,z)} > \ge}}\\
 &= \limsup_{R \to 0} R^{2-n} \mathcal{P}\prs{ P_r(z) \cap \prs{\Sigma
\backslash E_i} } \\
 &= 0.
 \end{align*}
The result follows.
\end{proof}
\end{lemma}

\begin{lemma} \label{wktostronglemma17} One has that $\sqg{\N_t^i}$ does not
converge
to $\N_t$ strongly in $H^{1,2}_{loc}$ if and only if $\mathcal{P}^{n-2}(\Sigma)
> 0$
and $\nu\prs{M \times [-1,0]} > 0$.

\begin{proof}
It follows from Lemma \ref{weaktostronglem10} that if $\mathcal{P}^{n-2}(\Sigma)
>
0$ then for $\mathcal{P}^{n-2}$ almost everywhere $z \in \Sigma$ one has
\begin{equation*}
\Theta\prs{\nu,z} = \Theta \prs{\mu,z} \geq \ge_0,
\end{equation*}
hence $\nu\prs{M \times [-1,0]} = \nu\prs{\Sigma} >0$, and
$\tfrac{1}{2}\brs{F_{\N_t^i}}^2 \, dV \, dt$ does not converge to
$\tfrac{1}{2}\brs{F_{\N_t}}^2 \, dV \, dt$. Therefore $\sqg{\N_t^i}$ doesn't
converge to
$\N_t$ strongly in $H^{1,2}_{loc}$. 
Conversely, directly from the definition of $\nu$, if $\nu\prs{M \times
[-1,0] } > 0$ then $\{\N^i_t\}$ cannot converge strongly to $\N$ in
$H^{1,2}_{loc}$.
\end{proof}
\end{lemma}

\section{Tangent measures and stratification} \label{sec:tmstrat}

In this section we establish results on the structure of tangent measures along
Yang-Mills flow which will be central in the sequel.  First we discuss the space
$T_z\mu$ of all tangent measures of $\mu$ for $z \in \Sigma$. We first show that
every tangent measure is invariant under parabolic dilations.  Building upon
this, we will associate to each tangent measure a nonnegative integer which is
the
dimension of the largest parabolic dilation invariant subspace which is a subset
of the points of maximal density.  Using this dimension we can then stratify the
set $\Sigma$ accordingly. In
particular, we demonstrate enough structure on the tangent measures to apply a
stratification result of White \cite{White}, which generalizes Federer's
dimension reduction argument \cite{FZ}.

\subsection{Setup}

For the following we set
\begin{equation*}
\bR_+^{n+1} : = \bR^n \times [0,\infty), \quad \bR_-^{n+1} : = \bR^n \times
(-\infty,0].
\end{equation*}
\begin{defn}
For $z_0 = \prs{x_0, t_0} \in \bRn \times \bR$ and $\la >0$, define
\emph{parabolic dilation} and \emph{Euclidean dilation} respectively by,
\begin{align*}
\mathsf{P}_{z_0,\la}(x,t) &:= \prs{\gl(x-x_0), \gl^2 (t-t_0)},\\
\mathsf{D}_{x_0,\la}(x) &:= \gl(x-x_0).
\end{align*}
Moreover, we may apply parabolic rescaling to a measure as follows. For all $A
\subset \bRn \times \bR$, we have
\begin{align*}
\mathsf{P}_{z_0,\la}(\mu)(A) :=&\ \gl^{2-n} \mu \prs{\mathsf{P}_{z_0,\la}(A)},\\
\mathsf{D}_{x_0,\gl}\prs{\mu} \prs{A} :=&\ \gl^{4-n} \mu
\prs{\mathsf{D}_{x_0,\gl}
A}.
\end{align*}
We note that this scaling law reflects the scaling properties for Yang-Mills
flow densities, and not a pure parabolic rescaling of say Euclidean measure.
\end{defn}
\begin{defn} For any $z_0 \in \Sigma$, the \emph{tangent measure cone of $\mu$
at $z_0$}, $T_{z_0}(\mu)$, consists of all nonnegative Radon measures on
$\mathbb{R}^{n+1}$ that are given by
\begin{equation*}
T_{z_0}(\mu) := \left\{ \mu^* \mid \exists r_i \to 0, \text{ such that
}\mathsf{P}_{z_0,r_i}(\mu) \to \mu^* \right\}.
\end{equation*}
\end{defn}
Fixing $z_0 \in \Sigma$ and $\mu^* = \mu^*_s ds \in T_{z_0}(\mu)$, we set,
for any $z = (x,t) \in \mathbb{R}^{n+1}$,
\begin{equation*}
\Theta \prs{\mu^*, z , r} := r^4 \int_{M \times \{{t-r^2}\}} G_z(y,s) \, d
\mu_s^*(y).
\end{equation*}
This is monotonically nondecreasing with respect to $r$ so that the
\emph{$\mu^*$
density at $z$}, given by
\begin{equation*}
\Theta(\mu^*,z) := \lim_{r \to 0} \Theta\prs{\mu^*,z,r},
\end{equation*}
exists and is upper semicontinuous for $z = (x,t) \in \mathbb{R}^{n+1}$.
Moreover, for any $z_0 \in \Sigma$ and $\mu^* \in T_{z_0}(\mu)$, we set
\begin{align*}
U\prs{\Theta\prs{\mu^*}} &:= \left\{ z \in
\bR^{n+1} \mid \Theta\prs{\mu^*,z} = \Theta\prs{\mu^*,0} \right\},\\
V \prs {\Theta\prs{\mu^*}} &:= U\prs{\Theta\prs{\mu^*}} \cap
\prs{\bR^{n} \times \{0\} },\\
W \prs{\Theta\prs{\mu^*}} &:= \left\{(x,0) \in \bRn \times \bR \mid \forall
(y,s) \in \bR^{n+1}_- ,\Theta \prs{\mu^*, \prs{y,s}} = \Theta
\prs{\mu^*,\prs{x+y,s}} \right\}.
\end{align*}
\begin{defn} For $z_0 \in \Sigma$ and $\mu^* \in T_{z_0}(\Sigma)$, let
\begin{equation*}
\dim \prs{\Theta \prs{\mu^*}} = 
\begin{cases}
\dim \prs{V\prs{\Theta\prs{\mu^*}}} + 2, &\text{ if }
U\prs{\Theta(\mu^*)} = V\prs{\Theta \prs{\mu^*}} \times
\mathbb{R},\\
\dim \prs{V\prs{\Theta\prs{\mu^*}}} &\text{ otherwise.}
\end{cases}
\end{equation*}
\end{defn}

\subsection{Preliminary results}

In this subsection we show various preliminary results on the structure of
tangent measures.  First we establish the existence of at least one tangent
measure in Lemma \ref{tangmeasexist}.  We then establish parabolic scaling
invariance of tangent measures in Lemma \ref{lem:8.3.2LWYM}.  

\begin{lemma}  \label{tangmeasexist} Given a weak
limit measure $\mu$, $z_0 \in \Sigma$, and $\gl_i \to 0$ there exists a
subsequence $\{ \la_{i_j} \}$ and some nonnegative Radon measure $\mu^*$ on
$\mathbb{R}^{n+1}$ such that $\mathsf{P}_{z_0,\la_{i_j}}(\mu) \to \mu^*$ as weak
convergence of Radon measures on $\mathbb{R}^{n+1}$. 
\begin{proof} We fix some small radius $r_0$ and claim that
\begin{align} \label{tangmeasexist10}
\sup_{(z,r) \in M \times [-1,0] \times (0,r_0)} r^{2-n} \, \mu(P_r(z)) < \infty.
\end{align}
In particular, we use a change of variables and Theorem \ref{thm:8.2.2LWYM} to
yield
\begin{align*}
r^{2-n} \mu(P_r(z)) =&\ r^{2-n} \lim_{i \to \infty} \int_{P_r(z)}
\brs{F_{\N^i}}^2 \, dV \, dt\\
=&\ \lim_{i \to \infty} r^{2-n} \int_{t=0}^{r^2} \int_{S_{\sqrt{t}}}
\brs{F_{\N^i}}^2 \phi \, dV \, dt\\
=&\ \lim_{i \to \infty} r^{2-n} \int_{s=0}^r s \int_{S_{s}} \brs{F_{\N^i}}^2
\phi \, dV \, ds\\
\leq&\ \lim_{i \to \infty} C r^{-2} \int_{s=0}^r s \Phi(s) \, ds\\
\leq&\ \lim_{i \to \infty} C (\Phi(r_0)) r^{-2} \int_{s=0}^r s \, ds\\
\leq&\ C.
\end{align*}
Hence, using (\ref{tangmeasexist10}), for any $\gl_i$ the sequence of dilated
measures $\mathsf{P}_{z_0,\gl_i}(\mu)$ is
uniformly bounded on all Borel sets in $\mathbb R^{n+1}$, hence by the weak
compactness of families of uniformly bounded Radon measures we obtain the
existence of the subsequential limiting measure $\mu$.
\end{proof}
\end{lemma}

\begin{lemma}\label{tangmeasureconelemma}  For any
$z_0 \in \Sigma$, $0 < r_1 < r_2 < \infty$ a sequence $\gl_i \to 0$ and a
blowup
sequence $\bar{\N}^i_t$ one has
\begin{equation*}
\lim_{i \to \infty} \int_{-r_1^2}^{-r_2^2} \int_{\bRn} \brs{x \hook
F_{\bar{\N}_t^i} + 2t \del_t \bar{\N}_t^i}^2 G_{z_0} \, dx \, dt = 0.
\end{equation*} 

\begin{proof} First recall that as
convergence of Radon measures on $\bRn$ we have
\begin{align*}
\tfrac{1}{2}\brs{{\overline{F_t}^i}}^2\, dV \to \mu_t^* \mbox{ for all } t \in
\mathbb (-\infty,0].
\end{align*}
Hence, for any $R > 0$, applying a change of variables we obtain
\begin{gather} \label{tangmeascone10}
\begin{split}
R^4 \int_{\bRn \times \{- R^2\}} G_{(0,0)} \, d\mu_t^* \, dt
&= \lim_{i \to \infty} \int_{\bRn \times \{- R^2\}}
\tfrac{R^4}{2}\brs{{\overline{F}_t^i}}^2 G_{(0,0)} \, dV_x \, dt \\
&= \lim_{i \to \infty} \int_{\bRn \times \{- R^2\}}  \tfrac{R^4 \la_i^4}{2
}\brs{{F_{t_0 + \la_i^2 t}^i}(x_0 + \la_i x)}^2 G_{(0,0)} \, dV_x \, dt \\
&=\lim_{i \to \infty}\int_{\bRn \times \{t_0- R^2\la_i^2\}}  \tfrac{(\la_i
R)^4}{2 \la_i^2}\brs{{F_{t}^i}}^2 G_{z_0} \, dV_y \, ds\\
&= \lim_{\la_i \to 0} \left[ \left. \int_{\bRn}  (\la_i R)^4 \, \mu_t
\right|_{t = t_0- R^2\la_i^2} \right]\\
&= \Theta\prs{\mu,z_0},
\end{split}
\end{gather}
where the last line follows from Lemma \ref{lem:cute}.
In particular, the $\Phi$ functional is approximately constant in $R$ for the
connections $\bar{\N}^i_t$, and hence using \eqref{Phimono} we obtain the
result.
\end{proof}
\end{lemma}

For $\Omega \subset \bR^n \times \bR$ we will use $\left. \mu^* \right\lfloor
{\Omega}$ to
denote the restriction of the tangent measure to $\Omega$.

\begin{lemma} \label{lem:8.3.2LWYM} For any $z_0 \in \Sigma$ and $\mu^* \in
T_{z_0}(\mu)$, the quantity $\left. \mu^* \right\lfloor {\mathbb{R}^{n+1}_-}$ is
invariant under all parabolic dilation, i.e.
\begin{equation*}
\mathsf{P}_{\gk}\prs{\left. \mu^* \right\lfloor {\mathbb{R}^{n+1}_-}} = \left.
\mu^* \right\lfloor {\mathbb{R}^{n+1}_-}.
\end{equation*}
\begin{proof} First we observe that
\begin{align*}
\mathsf{P}_{\gk}\prs{\left. \mu^* \right\lfloor {\mathbb{R}^{n+1}_-}}
&= \mathsf{P}_{\gk}\prs{ \left\{ \prs{\mu^*_t,t} \mid t \in (-\infty,0]
\right\} }\\
&= \left\{\prs{ \mathsf{D}_{\gk}(\mu_t^*), \gk^2 t } \mid t \in (-\infty,0]
\right\} \\
&=\left\{\prs{ \mathsf{D}_{\gk}\prs{\mu_{\tfrac{t}{\gk^2}}^*}, \gk^2 t }\mid t
\in
(-\infty,0] \right\}.
\end{align*}
Thus, to prove the lemma it suffices to show that for all $\gk < 0$, for all $t
\in (-\infty,0]$,
\begin{equation*}
\mathsf{D}_{\gk} \prs{\mu^*_{\tfrac{t}{\gk^2}}} = \mu^*_t.
\end{equation*}
Since $\gk$ is arbitrary this is equivalent to demonstrating this at $t=-1$.  To
prove this it suffices to show the result for $\mu_t^*$ multiplied by an
arbitrary smooth positive function.  We will take advantage of this by inserting
a factor of the Greens function $G = G_{(0,0)}$, then multiplying by an
arbitrary compactly supported positive function.  This will allow us to take
advantage of monotonicity formulae to obtain the result.  In particular, we will
show that
\begin{equation}\label{eq:8.27}
\gk^{n-4} \int_{\bRn} \phi(\gk x) G\prs{\gk x, -1} d\mu^*_{-\gk^{-2}} =
\int_{\bRn} \phi(x) G(x,-1) \, d\mu_{-1}^*,
\end{equation} 
for any $\phi \in C_0^1(\bRn)$.
We attain the claim \eqref{eq:8.27} if we can show that
\begin{equation}\label{eq:lambdaderivvanish}
\lim_{i \to \infty} \frac{d}{d \gk} \brk{\frac{\gk^{n-4}}{2}\int_{\bRn}
\phi(\gk x) G\prs{\gk x, -1} \brs{\overline{F}_{- \gk^{-2}}^i}^2 \, dx } =
0.
\end{equation}
For notational simplicity we will remove both the sequence index $i$ and the bar
from the connection. Manipulating
the integrand by applying the change of coordinates $\gk x = y$ yields,
\begin{align*}
\frac{\gk^{n-4}}{2} \int_{\bRn \times \sqg{-1}} \phi(\gk x) & \brs{F_{-\gk^{-2}}
\prs{x}}^2  G(\gk x,-1) \, dx \\
&= \frac{\gk^{n-4}}{2} \int_{\bRn \times \sqg{-1}} \phi(y)
\brs{F_{-\gk^{-2}}\prs{\tfrac{y}{\gk}}}^2 G(y,-1) \, d \prs{\frac{y}{\gk}}\\
&= \frac{\gk^{-4}}{2} \int_{\bRn \times \sqg{-1}} \phi(y)
\brs{F_{-\gk^{-2}}\prs{\tfrac{y}{\gk}}}^2 G(y,-1) \, dy\\
&= \left[ \Phi \prs{\tfrac{1}{\gk};\N_t} \right|_{t=-1}.
\end{align*}
Set $R(\gk) : = \frac{1}{\gk}$. Then by a calculation similar to
\eqref{eq:delRmonotonicity}, where the final term vanishes since the cutoff
function $\phi$ no longer depends on the parameter $R$, we see
that
\begin{align*}
\frac{\del}{\del \gk} \brk{\Phi \prs{\tfrac{1}{\gk};\N_t}}
&= \frac{-1}{\gk^2} \frac{\del}{\del R} \brk{\Phi \prs{R(\gk);\N_t}}\\
&= \frac{-1}{\gk^5} \int_{S_{\gk^{-1}}}
\brs{t}\brs{\frac{x}{t}\hook F - 2 D^* F}^2 \phi G \, dx + \frac{4}{\gk^5}
\int_{S_{\gk^{-1}}} \ip{ \prs{x \hook F - 2 t \prs{D^* F} }, \N \phi \hook F } 
G \, dx
\end{align*}
Taking the limit as $i \to \infty$, we have that the first quantity vanishes by
Lemma \ref{tangmeasureconelemma}. For the second we apply weighted H\"{o}lder's
inequality for an arbitrary $\ge > 0$,
\begin{align*}
\frac{1}{\gk^5}& \int_{S_{\gk^{-1}}} \ip{ \prs{x \hook F - 2 t \prs{D^* F} }, \N
\phi \hook F }  G \, dx\\
&\leq \frac{C}{\ge \gk^5} \int_{S_{\gk^{-1}}} \brs{ \prs{x \hook F - 2 t
\prs{D^* F} } }^2 G \, dx  + \frac{\ge}{\gk^5} \int_{S_{\gk^{-1}}} \brs{ \N
\phi}^2 \brs{F }^2  G \, dx.
\end{align*}
The first factor vanishes with another application of
Lemma \ref{tangmeasureconelemma}.  The integrand of the second term is bounded
by the monotonicity of $\Phi$, using an argument similar to
(\ref{tangmeascone10}).  Sending $\ge \to 0$ therefore yields
\eqref{eq:lambdaderivvanish}. The result follows.
\end{proof}
\end{lemma}

\subsection{Stratification of tangent measures}

\begin{lemma}\label{lem:8.3.3LWYM}  For $z_0 \in
\Sigma$ and $\mu^* \in T_{z_0}(\mu)$, the following hold.
\begin{enumerate}
\item For all $z \in \mathbb{R}^{n+1}$, $\Theta(\mu^*,z) \leq
\Theta(\mu^*,0)$.
\item If $z \in \mathbb{R}^{n+1}$ satisfies $\Theta\prs{\mu^*,z} =
\Theta\prs{\mu^*,0}$, then for all $\la >0$ and $v \in
\mathbb{R}^{n+1}_{-}$,
\begin{equation*}
\Theta\prs{\mu^*, z+v} = \Theta\prs{\mu^*, z + \mathsf{P}_{\la}v}.
\end{equation*}
\end{enumerate}
\begin{proof} For $\mu^* \in T_{z_0}(\mu)$, there exists some sequence $r_i
\to 0$ such that $\mathsf{P}_{z_0,r_i}(\mu) \to \mu^*$.  We first observe how
the rescaling law for $\Phi$ is reflected in the definition of $\Theta$.  In
particular, since we are integrating over a space slice we apply the scaling law
for $\mathsf{D}_{\gl}$ and change variables to yield
\begin{align*}
 \Theta(\mathsf{P}_{\gl}(\mu),z,r) =&\ \frac{r^4}{2} \int_{S_r} G_z
\mathsf{P}_{\gl}(\mu)\\
 =&\ \frac{r^4}{2} \int_{S_r} \left( \gl^n \mathsf{P}_{\gl}^*
G_{\mathsf{P}_{\gl}(z)} \right) \left( \gl^{4-n} \mathsf{P}_{\gl}^* \mu
\right)\\
 =&\ \frac{(\gl r)^4}{2} \int_{\mathsf{P}_{\gl}(S_r)} G_{\mathsf{P}_{\gl}(z)}
\mu\\
 =&\ \Theta(\mu,\mathsf{P}_{\gl}(z), \gl r).
\end{align*}
Using this, for any $r>0$, and
$z = \prs{x,t} \in \bR^{n+1}$,
\begin{align}
\begin{split}\label{eq:Thetanmu0ineq}
\Theta\prs{\mu^*,z}
&\leq \Theta\prs{\mu^*,z,r}\\
&= \lim_{r_i \to 0} \Theta \prs{ \mathsf{P}_{z_0,r_i}(\mu),z,r} \\
&= \lim_{r_i \to 0} \Theta\prs{\mu,z_0 + \prs{r_i x,r_i^2 t}, r_i r} \\
&\leq \Theta\prs{\mu,z_0}\\
&=\Theta\prs{\mu^*,0},
\end{split}
\end{align}
where we have applied the upper semicontinuity of $\Theta(\mu,\cdot,\cdot)$
with respect to the last two variables. Thus claim (1) follows.

To prove claim (2), observe that the hypothesis $\Theta(\mu^*,z) =
\Theta(\mu^*,0)$ implies that the inequalities of \eqref{eq:Thetanmu0ineq} are
equalities.  This implies that $\Theta\prs{\mu^*,z,r}=\Theta(\mu,z_0)$,
namely, it is constant with respect to $r$. By an argument similar to that of
Lemma \ref{lem:8.3.2LWYM}, we have that $\Theta\prs{\mu^*, z+v} =
\Theta\prs{z+\mathsf{P}_{\la}(v)}$ for any $v \in \mathbb{R}^{n+1}_-$ and $\la
>0$. The result follows.
\end{proof}
\end{lemma}

\begin{prop} \label{prop:8.3.4LWYM}  For $z_0 \in
\Sigma$ and $\mu^* \in T_{z_0}(\mu)$,
\begin{equation*}
V\prs{\Theta \prs{\mu^*, \cdot}} = W \prs{\Theta \prs{\mu^*, \cdot}}.
\end{equation*}
In particular, both $V\prs{\Theta \prs{\mu^*, \cdot}}$ and $W\prs{\Theta
\prs{\mu^*, \cdot}}$ are linear subspaces of $\bR^n$. Moreover,
$U\prs{\Theta\prs{\mu^*,\cdot}}$ is either $V\prs{\Theta\prs{\mu^*,\cdot}}$,
or $V\prs{\Theta\prs{\mu^*,\cdot}} \times (-\infty,a]$ for some $0 \leq a \leq
\infty$ and $\Theta\prs{\mu^*,\cdot}$ is time-independent on $(-\infty,a]$.

\begin{proof}
First we show that $W(\Theta(\mu^*, \cdot)) \subset V(\Theta(\mu^*,\cdot))$.
 Fix $(x,0) \in W\prs{\Theta \prs{\mu^*, \cdot}}$. Since the second component
is identically zero it suffices to verify that $(x,0) \in
U\prs{\Theta\prs{\mu^*,\cdot}}$. Note that by definition of $W\prs{\Theta
\prs{\mu^*, \cdot}}$, choosing $y = -x$ as in its definition,
\begin{align*}
\Theta\prs{\mu^*,(x,0)}
&= \Theta\prs{\mu^*,(x - x,0)}= \Theta\prs{\mu^*,0}.
\end{align*}
It follows that $W\prs{\Theta \prs{\mu^*, \cdot}} \subset V\prs{\Theta
\prs{\mu^*, \cdot}}$.

Now we show the containment $V\prs{\Theta \prs{\mu^*, \cdot}} \subset
W\prs{\Theta \prs{\mu^*, \cdot}}$.  First note that $V\prs{\Theta
\prs{\mu^*, \cdot}}$ is closed under scalar multiplications from Lemma
\ref{lem:8.3.2LWYM}.  Next, for any nonzero $x \in
V\prs{\Theta\prs{\mu^*,\cdot}}$ we have that for all $\la > 0$ and all $v \in
\bR^{n+1}_{-}$, by applying Lemma \ref{lem:8.3.3LWYM} (2), and using the
parabolic scaling invariance of $\Theta$ from Lemma \ref{lem:8.3.2LWYM},
\begin{align}
\begin{split}\label{eq:8.33LW}
\Theta \prs{\mu^*, \prs{x,0} + v}
&= \Theta \prs{\mu^*, \prs{x,0}+\mathsf{P}_{\la} v}\\
&=\Theta\prs{\mu^*,\mathsf{P}_{\la^{-1}}\prs{\prs{x,0} + \mathsf{P}_{\la} v}}\\
&=\Theta\prs{\mu^*,\mathsf{P}_{\la^{-1}}\prs{x,0} + v}.
\end{split}
\end{align}
By the upper semicontinuity of $\Theta$, sending $\gl \to \infty$ yields
\begin{align*}
\Theta(\mu^*, (x,0) + v) \leq \Theta(\mu^*, v).
\end{align*}
On the other hand, since $v - \mathsf{P}_{\la^{-1}}\prs{x,0} \in \bR^{n+1}_-$,
we can replace $v
\mapsto v - \mathsf{P}_{\la^{-1}}\prs{x,0}$ throughout the equalities in
\eqref{eq:8.33LW} and obtain that
\begin{equation*}
\Theta\prs{\mu^*,\prs{x,0} + v - \mathsf{P}_{\la^{-1}}\prs{x,0}} =
\Theta\prs{\mu^*, v}.
\end{equation*}
Again sending $\gl \to \infty$ and utilizing the upper semicontinuity of
$\Theta\prs{\mu^*,\cdot}$ yields
\begin{align*}
\Theta(\mu^*, (x,0) + v) \geq \Theta(\mu^*, v).
\end{align*}
Hence we have $\Theta\prs{\mu^*, v} = \Theta \prs{\mu^*, \prs{x,0} + v}$, and so
we conclude $V\prs{\Theta\prs{\mu^*,\cdot}} \subset
W\prs{\Theta\prs{\mu^*,\cdot}}$ so that $V\prs{\Theta\prs{\mu^*,\cdot}} =
W\prs{\Theta\prs{\mu^*,\cdot}} $.

Note that by definition of $W\prs{\Theta\prs{\mu^*,\cdot}} $ we have that it
is closed under linear combinations since for all $\prs{x,0}$, $\prs{v,0}$ in
$W\prs{\Theta\prs{\mu^*,\cdot}}$ we have that for all $\prs{y,s} \in
\bR_-^{n+1}$, just iterating its definition twice
\begin{align*}
\Theta \prs{\mu^*,  \prs{\prs{x+v} + y, s}}
&= \Theta \prs{\mu^*,  \prs{x + y, s}} \\
&= \Theta \prs{\mu^*,  \prs{y, s}}.
\end{align*}
Therefore by equality of $V\prs{\Theta\prs{\mu^*,\cdot}}$ to
$W\prs{\Theta\prs{\mu^*,\cdot}}$, with the combined scaling invariance and
linear combinations invariance both are linear subspaces of $\bRn$.

Now we prove the remaining statement of the proposition concerning the structure
of $U\prs{\Theta\prs{\mu^*,\cdot}}$. 
 Suppose that $z := \prs{x,t} \in U\prs{\Theta\prs{\mu^*,\cdot}}$ with $t <
0$. Then for all $w := \prs{y,s} \in \bR^{n+1}$ with $s \leq t$ and for all $\la
> 0$, using Lemma \ref{lem:8.3.3LWYM} (b)
\begin{align}
\begin{split}\label{eq:8.3.4LWYMc}
\Theta \prs{\mu^*, \mathsf{P}_{\la^{-1}}(w)}
&= \Theta \prs{\mu^*, w}\\
&= \Theta \prs{\mu^*, z + w - z}\\
&= \Theta \prs{\mu^*, z + \mathsf{P}_{\la^{-1}}(w-z)}.
\end{split}
\end{align}
In particular, take $\la \in \prs{0,1}$, and note that consequently
$\tfrac{s}{\la^{2}} \leq s \leq t$. So taking \eqref{eq:8.3.4LWYMc} and
replacing $w \mapsto \mathsf{P}_{\la}\prs{w}$ in  yields
\begin{equation}\label{eq:8.3.4LWYMb}
\Theta \prs{\mu^*,w} :=\Theta \prs{\mu^*, z+w-\mathsf{P}_{\la^{-1}}(z)}.
\end{equation}
Taking $\la \to 0$, we see that $\Theta\prs{\mu^*,w} \leq
\Theta\prs{\mu^*,z+w}$. Taking \eqref{eq:8.3.4LWYMb} again and instead
replacing $w \mapsto w + \mathsf{P}_{\la^{-1}}(z)$, we conclude that
\begin{equation*}
\Theta\prs{\mu^*,w+\mathsf{P}_{\la^{-1}}\prs{z}} = \Theta\prs{\mu^*,z+w}.
\end{equation*}
Again sending $\la \to 0$ we obtain that
\begin{equation*}
\Theta\prs{\mu^*,z} \leq \Theta\prs{\mu^*,z+w}.
\end{equation*}
Therefore, we conclude that for any $z := \prs{x,t} \in
U\prs{\Theta\prs{\mu^*,\cdot}}$ with $t < 0$, for all $w :=
\prs{y,s}$ with $s \leq t$,
\begin{equation}\label{eq:8.3.4LWYMa}
\Theta\prs{\mu^*,w} = \Theta\prs{\mu^*,z+w}.
\end{equation}
Then choosing $w \equiv z$, iterating \eqref{eq:8.3.4LWYMa}, applying the
parabolic
scaling invariance of $\Theta$ from Lemma \ref{lem:8.3.2LWYM}, and the upper
semicontinuity of $\Theta \prs{\mu^*, \cdot }$, one has
\begin{align*}
\Theta\prs{\mu^*,0} &=\ \Theta(\mu^*, z) = \Theta(\mu^*, z + z) = \dots =
\Theta(\mu^*, mz)\\
&=\Theta \prs{\mu^*, \prs{mx,mt}}\\
&= \Theta \prs{\mu^*, \mathsf{P}_{\frac{1}{m}}\prs{mx,mt}}\\
&= \Theta \prs{\mu^*, \prs{x,\tfrac{t}{m}}}\\
&\leq \Theta \prs{\mu^*, \prs{x,0}}.
\end{align*}
Combining this with Lemma \ref{lem:8.3.3LWYM} (1) we conclude that $\prs{x,0}
\in V\prs{\Theta\prs{\mu^*,\cdot}} = W\prs{\Theta\prs{\mu^*,\cdot}}$.
Therefore
\begin{align*}
\Theta \prs{ \mu^*, \prs{0,t}} = \Theta(\mu^*, (x,0) + (0,t)) =
\Theta(\mu^*,0)
\end{align*}
It follows that $\prs{0,t}\in U\prs{\Theta\prs{\mu^*,\cdot}}$.  It follows
that $\Theta\prs{\mu^*,\cdot}$ is actually time independent for $t \leq 0$.
Therefore for all $t \leq 0$,
\begin{equation*}
V \prs{\Theta \prs{\mu^*,\cdot}} := U \prs{\Theta \prs{\mu^* , \cdot }} \cap
\prs{\bR^n \times \sqg{t}}.
\end{equation*}
Lastly, if $z = (x,t) \in U \prs{\Theta\prs{\mu^*,\cdot}}$ with $t > 0$, then we
can repeat the argument above to show that
$\Theta\prs{\mu^*,\cdot}$ is time-independent up to $t$. We set $a$ to be the
value of the maximal time $t \geq 0$ for which this time independence exists on.
Then we have $U \prs{\Theta\prs{\mu^*,\cdot}} = V\prs{\Theta
\prs{\mu^*,\cdot}} \times (-\infty, a]$, which concludes the proof.
\end{proof}
\end{prop}

We can now establish Theorem \ref{dimredthm}, which we restate for convenience.

\begin{thm*} 
For $0 \leq k \leq n-2$ let
\begin{equation*}
\Sigma_k = \left\{ z_0 \in \Sigma \mid \dim \prs{ \Theta \prs{\mu^*, \cdot}}
\leq
k, \forall \mu^* \in T_{z_0}(\mu) \right\}.
\end{equation*}
Then $\dim_{\mathcal{P}}\prs{\Sigma_k} \leq k$ and $\Sigma_0$ is countable.
\begin{proof}[Proof of Theorem \ref{dimredthm}] This is a direct consequence of
(\cite{White} Theorem 8.2).  To connect directly to the notation of that paper,
the function $f$ is given by the density function.  Hypothesis (1), the
subsequential compactness of blowup limits, is established in Lemma
\ref{tangmeasexist}.   Hypothesis (2) is clear from the construction of blowup
limits.  Hypothesis (3), the parabolic scaling invariance of the limit
functions, is established in Lemma \ref{lem:8.3.2LWYM}.  The theorem thus
applies to give the claimed statement.
\end{proof}
\end{thm*}

\section{Characterization of strong convergence}\label{s:charstrconv}

In this section we prove Theorem \ref{thm:weaktostrong} (stated more precisely
as Theorem \ref{weaktostrongH12} below), which characterizes when the weak
convergence in $H^{1,2}$ for
sequences as in Theorem \ref{fullweakcompactthm} can be improved to strong
convergence.  In particular, we know this means that the defect measure is
nontrivial, and we use this to obtain refined estimates on tangent measures,
eventually leading to a further blowup sequence which yields the required
Yang-Mills connection.

\begin{thm} \label{weaktostrongH12} Suppose $\{\N^i_t\}$ is a sequence of smooth
solutions to Yang-Mills flow on $[-1,0]$ with 
\begin{align*}
\sup_i \int_{M \times [-1,0]} \left( \brs{\frac{\del \N^i_t}{\del t} }^2 +
\brs{F_{\N_t^i}}^2 \right) dV dt < \infty.
\end{align*}
Furthermore, suppose $\{\N^i_t\} \to \N^{\infty}_t$ weakly in $H^{1,2}_{loc}$. 
Then
exactly
one of the following holds:
\begin{itemize}
\item There exists a blowup sequence converging to a
Yang-Mills connection on $S^4$.
\item One has
\begin{align*}
\brs{F_{\N^i_t}}^2 \, dV \, dt \to \brs{F_{\N^{\infty}_t}}^2 \, dV \, dt
\end{align*}
as convergence of Radon measures, and hence $\{\N^i_t\} \to \N^{\infty}_t$
strongly
in $H^{1,2}_{loc}$.  Thus $\N^{\infty}_t$ is a weak solution of Yang-Mills flow
satisfying
$\mathcal P^{n-2}(\Sigma) = 0$.
\end{itemize}
\begin{proof} 

We adopt the setup of the previous sections in this proof.  In particular, we
assume we have a particular blowup sequence together with a limiting tangent
measure $\mu^*$.  Moreover, various results from \S \ref{limmeasstruct} were
established which apply to almost every point in the singular set.  We will
assume without loss of generality that our tangent measure arises from a blowup
sequence around one of these points, so that the Lemmas of \S
\ref{limmeasstruct} apply.  In particular, in the discussion below we will refer
to a sequence $\{\N_t^i\}$ but this will refer to a \emph{blowup} sequence, not
the original given sequence of the statement.

\begin{lemma}  \label{mustarmeasure} For $t \in
(-4,0]$, we have $\mathcal{H}^{n-4} \brk{\Sigma_t^*} > 0$.
\begin{proof}
Suppose to the contrary there is some $t_0
\in (-4,0]$ such that $\mathcal H^{n-4}(\Sigma_{t_0}^*) = 0$. Then for all $\ge
> 0$, there exists some $\delta_{\ge} >0$ and a covering of $\Sigma_{t_0}^*$ of
the form $\{ B_{r_j}(x_j) \}_{i \in \mathbb{N}}$, with $x \in \Sigma_{t_0}^*$
and $0 < r_j \leq \delta_{\ge}$ satisfying
\begin{equation*}
\sum_{j=1}^{\infty}r_j^{n-4} < \ge.
\end{equation*}
Now, because $\mu^*_{t_0}\brk{ B_1 \backslash \prs{\bigcup_{j \in \mathbb{N}}
B_{r_j}(x_j) }} = 0,$ then by a diagonalization argument we may choose a
subsequence $\{ \N^{i}_t \}$ such that
\begin{equation} \label{mustarmeasure10}
\lim_{{i} \to \infty} \tfrac{1}{2}\int_{B_1 \backslash \bigcup_{j \in
\mathbb{N}} B_{r_j}(x_j)} \brs{ F_{\N^{i}_{t_0}}}^2 \, dV = 0.
\end{equation}

Furthermore we will use (\ref{Phimono}) to estimate the curvature on balls in
the cover.  We choose a cutoff function $\phi$ for a ball of radius $1$, and
further fix some radius $R$.  Note that for the compact set $\supp \phi$ there
is a uniform estimate for the $L^2$ norm of the Yang-Mills energy.  This follows
from the argument of Lemma \ref{tangmeasexist}, which shows that the sequence of
blowup measures is uniformly locally finite.  In particular, there exists $K <
\infty$ such that
\begin{align*}
\int_{\supp \phi} \brs{F_{\N^i_{t_0}}}^2 dV \leq K.
\end{align*}
Hence using (\ref{Phimono}) we estimate for all $i,j \in \mathbb N$,
\begin{align*}
\tfrac{r_j^{4-n}}{2} \int_{B_{r_j}(x_j)} \brs{F_{\N_{t_0}^{i}}}^2 \, dV
&\leq \tfrac{1}{2 e} r_{j}^4 \int_{B_{r_j}(x_j) \times\{ (t_0 + r_j^2) - r_j^2
\}}
\brs{F_{\N^{i}_t}}^2 \phi^2 G_{(x_i,t_0 + r_j^2)} dV\\
&\leq \tfrac{1}{2 e} \Phi\prs{\N_t^{i}, \prs{x_j, t_0 + r_j^2}, r_j} \\
&\leq \tfrac{1}{2 e} \Phi\prs{\N^{i}_t, \prs{x_j, t_0 + r_j^2}, R} + C K \prs{R
-
r_j} \\
&\leq \tfrac{R^{4-n}}{2e} \int_{\bRn \times \{t_0+r_j^2-R^2\}}
\brs{F_{\N_t^{i}}}^2 \phi^2 \, dV + CK \prs{R-r_j}\\
&\leq C\prs{R,K}.
\end{align*}
Therefore  we have that
\begin{align*}
\tfrac{1}{2}\int_{\bigcup_{j \in \mathbb{N}}B_{r_j}(x_j)}
\brs{F_{\N_{t_0}^{i}}}^2 dV
&\leq \tfrac{1}{2}\sum_{j\in \mathbb{N}} \int_{B_{r_j}(x_j)}
\brs{F_{\N_{t_0}^{i}}}^2 dV\\
&\leq C \sum_{j \in \mathbb{N}} r_j^{n-4}\\
&\leq C \ge.
\end{align*}
Choosing $\ge < \tfrac{\ge_0}{16C}$ and combining with (\ref{mustarmeasure10})
yields, for $i$ sufficiently large,
\begin{equation} \label{mustarmeasure30}
\int_{B_1} \brs{F_{\N^i_{t_0}}}^2 \, dV \leq \frac{\ge_0}{2}.
\end{equation}
Also, using Lemma \ref{wktostronglemma5} we have
\begin{equation} \label{mustarmeasure40}
\lim_{i \to \infty} \int_{P_2}\brs{\frac{\del \N_t^i}{\del t}}^2 dV = 0.
\end{equation}
Using Lemma \ref{lem:8.2.1LWYM} we find that for any $\phi \in
C_0^{\infty}(B_2)$ and $-4 < t_1 < t_2 \leq 0$ one has
\begin{equation} \label{mustarmeasure50}
\tfrac{1}{2} \int_{B_2}\prs{\brs{{F_{t_2}^i}}^2-\brs{{F_{t_1}^i}}^2} \phi \, dV
= - \int_{t_1}^{t_2}\int_{B_2}\prs{\brs{\del_t \N_t^i}^2 \phi + \ip{\N \phi
\hook {F_t^i},\del_t \N_t^i}} \, dV\, dt.
\end{equation}
Combining (\ref{mustarmeasure40})-(\ref{mustarmeasure50}) shows that the
limiting measure $\mu_t^*(\phi)$ is independent of time for $t \in (-4,0]$. 
Applying (\ref{mustarmeasure30}) again yields
\begin{align*}
\tfrac{1}{2}\int_{P_1}\brs{{F_t^i}}^2 dV \, dt
&\leq \tfrac{1}{2}\sup_{t \in \brk{-1,0}} \int_{B_1} \brs{{F_t^i}}^2 \, dV\\
&\leq \tfrac{1}{2} \int_{B_1} \brs{{F_{t_0}^i}}^2 \, dV + o(i)\\
&\leq \frac{\ge_0}{4}. 
\end{align*}
This is a contradiction to the assumption that $(0,0) \in \Sigma^*$. 
\end{proof}
\end{lemma}

\begin{prop} \label{singularsetisplane} There is a linear subspace $\mathscr{P}$
of dimension $\prs{n-4}$ such that for all $t < 0$ one has that $\supp(\mu^*_t)
= \mathscr{P}$
 \begin{proof} 
First, by Lemma \ref{wktostronglem15}, we have that $\Theta\prs{\nu,z} =
\Theta\prs{\mu,z}$ is $\mathcal{P}^{n-2}$-approximately continuous at $z_0$ for
$z
\in \Sigma$ and $\Theta\prs{\mu,z}$ is upper semicontinuous with respect to
$z$, we conclude that for $\mathcal{P}^{n-2}$-almost all $z \in
\Sigma^*$,
we have
\begin{equation*}
\Theta\prs{\mu^*,z} \geq \Theta\prs{\mu,z_0}.
\end{equation*}
Also, it follows from Lemma \ref{lem:8.3.3LWYM} that
\begin{equation*}
\Theta\prs{\mu^*,z} \leq \Theta\prs{\mu^*,0} = \Theta\prs{\mu,z_0}.
\end{equation*}
Hence for $\mathcal{P}^{n-2}$ a.e. $z \in \Sigma_*$,
\begin{equation} \label{singsetplaneloc10}
\Theta\prs{\mu^*,z} = \Theta\prs{\mu^*,0}.
\end{equation}
We now show that in fact all points in $\Sigma_*$ have maximal density.
In particular, by Proposition \ref{prop:8.3.4LWYM} there is some set
$\mathscr{S} \subset \bRn$ with $\mathcal{H}^{n-4}\prs{\mathscr{S}} = 0$ and an
$(n-4)$-dimensional plane $\mathscr{P} \subset \mathbb{R}^n$ such that
$\mathscr{S} \cap \mathscr{P} = \varnothing$, and
\begin{align*}
\Sigma_t^* = \mathscr{S} \cap \mathscr{P}
\end{align*}
for all $t$.  We claim that in fact $\mathscr{S} = \varnothing$. 
Suppose to the contrary we had some $z \in
\mathscr{S}$.  Note that by construction, it must hold that $0 < \Theta(\mu^*,z)
< \Theta(\mu^*,0)$.  By Lemma \ref{lem:8.3.3LWYM} (2) we have that for all $w
\in
\mathscr{P}$,
\begin{equation*}
\Theta(\mu^*,z) = \Theta(\mu^*, w + (z-w))= \Theta\prs{\mu^*,w +
\mathsf{P}_{\gl}(z - w)}.
\end{equation*}
Applying this for $w \in B_{\ge}^{n-4}(0) \subset \mathscr{P}$ and $\gl \in
[1-\ge,1]$ yields a set of positive $\PP^{n-2}$-measure in $\Sigma^*$ on which
$\Theta(\mu^*, \cdot) = \Theta(\mu^*, z) < \Theta(\mu^*,0)$, contradicting
(\ref{singsetplaneloc10}).
 \end{proof}
\end{prop}

Using this characterization of the singular set of the blowup limit, we can
refine our estimates on the blowup sequence to obtain further structure on the
blowup limit.  Without loss of generality we can assume that $\mathscr{P} =
\mathbb R^{n-4} \subset \mathbb R^n$ is the standard embedding in the first
$n-4$ coordinates, and we express a general point as $X = (x,y)$ where $x \in
\mathbb R^4, y \in \mathbb R^{n-4}$.  We first show two lemmas which give
improved vanishing results for the time derivative of the connection as well as
for the curvature in directions along the singular locus.

\begin{lemma}  \label{wktostronglemma10} Given the setup above and $0 < t_1 <
t_2
\leq 1$, one has
\begin{gather*}
\lim_{i \to \infty} \int_{-t_2}^{-t_1} \int_{B_1^{n}} \prs{\brs{ \frac{\del
\N^i_t}{\del t}}^2 + \sum_{j=1}^{n-4} \brs{\frac{\del}{\del y_j} \hook
F_{\N^i_t}}^2} \, dV \, dt = 0.
\end{gather*}
\begin{proof} We first observe that by rescaling the result of Lemma
\ref{wktostronglemma5} we observe that
\begin{align} \label{wktostronglemma1020}
\lim_{i \to \infty} \int_{-t_2}^{-t_1} \int_{B_1^{n}} \brs{ \frac{\del
\N^i_t}{\del t}}^2 \, dV \, dt = 0.
\end{align}
Now let $\xi_j = 2 \frac{\del}{\del y_j}$.  Since we know that the limiting
density $\Theta$ is a
multiple of the Hausdorff measure of the given $\mathbb R^{n-4}$ on each time
slice, applying the monotonicity formula (\ref{Psimono}) with centers $(X_0,t_0)
= (\prs{0,\xi_j},0)$ implies that for any $\rho > 0$ we have
\begin{align*}
0&= \lim_{i \to \infty} \int_{\rho}^1 r \int_M\int_{-4r^2}^{-r^2} 
\frac{\brs{\prs{X-\xi_j} \hook F_t^i + 2t \del_t {\N}^i_t}^2}{\brs{t}} \phi^2
G_{0,\xi_j} \, dV \, dt \, dr\\
& \geq \lim_{i \to \infty} C_{\rho} \int_{\rho}^1 r
\int_{B_1}\int_{-4r^2}^{-r^2}  \frac{\brs{\prs{X-\xi_j} \hook F^i_t + 2t \del_t
{\N}^i_t}^2}{\brs{t}} \, dV \, dt \, dr.
\end{align*}
Note that the second equality follows since, for a given $\rho$, on $B_1$ we
have that there is a constant $C_{\rho} \in (0,\infty)$ dependent solely on
$\rho$ such that $G_{0,\xi_j} \phi^2 \geq C_{\rho} \in (0,\infty)$.

Next we apply Fubini's theorem to switch the integration bounds $dr \, dt$ to
$dt \, dr$. In that case we have that if we set 
\begin{equation*}
I_i(r,t) := C_{\rho} r \int_{B_1} \frac{\brs{(X-\xi_j) \hook
F_{{\N}^i_t} + 2t \del_t {\N}^i_t }^2}{\brs{t}} \, dV
\end{equation*}
then applying Fubini's theorem to the regions corresponding to the variables $r$
and $t$ give that
\begin{align*}
0= \lim_{i \to \infty} \left[ \int_{- 4 \rho^2}^{- \rho^2}
\int_{\rho}^{\sqrt{-t}}  I_i(r,t) \, dr \, dt  + \int_{-1}^{- 4 \rho^2}
\int_{\tfrac{1}{4} \sqrt{-t}}^{\sqrt{-t}} I_i(r,t) \, dr \, dt + \int_{-4}^{-1}
\int_{\tfrac{1}{4} \sqrt{-t}}^1  I_i(r,t)  \, dr \, dt. \right]
\end{align*}
\color{black}
Then for $0 < t_1 < t_2 < 1$, if we choose $\rho \leq t_1$, so that $[-t_2^2, -
t_1^2] \subset \brk{ -1, - \rho^2}$ (the unions of the temporal domains of the
first two integrals) then we can conclude that since the sums are $0$ and the
arguments of each integral are positive,
\begin{align*}
0 &=\ \lim_{i \to \infty} \int_{-t_2^2}^{-t_1^2} \int_{B_1}\frac{\brs{(X-\xi_j)
\hook F_{{\N}_t^i} + 2t \del_t {\N}_t^i }^2}{\brs{t}} \, dV \, dt \\
&\geq\ \lim_{i \to \infty} \brs{t_1}^{-1} \int_{-t_2^2}^{-t_1^2} \int_{B_1}
\frac{1}{2} \brs{(X-\xi_j) \hook F_{\N_t^i}}^2 - C \brs{t \del_t
\N_t^i }^2 \, dV \, dt \\
&=\ \lim_{i \to \infty} \brs{t_1}^{-1} \int_{-t_2^2}^{-t_1^2} \int_{B_1}
\frac{1}{2} \brs{(X-\xi_j) \hook F_{{\N}_t^i}}^2\, dV \, dt.
\end{align*}
The second inequality follows using the Cauchy-Schwarz inequality, and the final
line follows from (\ref{wktostronglemma1020}).  Now observing that $\xi_j = 2
\frac{\del}{\del y_j}$ we see that for all $X = (x,y) \in B_1$ we have that
$\brs{\ip{(X -
\xi_j),\frac{\del}{\del y_j}}} \geq 1$.  The result follows.
\end{proof}
\end{lemma}

\begin{lemma} \label{wktostronglemma20}  Given the
setup above, there exists $(y,t) \in B_{1/2}^{n-4} \times
[-\frac{1}{2},-\frac{1}{4}]$ such that
\begin{align*}
\lim_{i \to \infty} & \sup_{0 < r \leq 1} r^{4-n} \int_{B_r^{n-4}(y)}
\int_{B_1^{4} \times \{y\} \times [-1,0]} \brs{\frac{\del \N^i_t}{\del t}}^2 \,
dx \, dt \, dy = 0,\\
\lim_{i \to \infty} & \sup_{0 < r \leq 1} r^{2-n} \int_{B_r^{n-4}(y) \times [t -
r^2,t]} \int_{B_1^{4}  \times \{y\}} \sum_{j=1}^{n-4} \brs{\frac{\del}{\del
y_j} \hook F_{\N^i_t}}^2 \, dx \, dt \, dy = 0.
\end{align*}
\begin{proof} To begin we show a preliminary statement using maximal functions. 
In particular, let
\begin{align*}
f_i
&: B_1^{n-4} \subset \bR^{n-4} \to [0,\infty), \qquad  f_i({y}) =
\int_{\prs{B_1^{4} \times \sqg{ {y}} }  \times \brk{-1,0}} \brs{\del_t
{\N}_i} ^2 \, dx\, dt \\
g_i &: B_1^{n-4} \times \brk{-1, -\tfrac{1}{8}} \to [0,\infty), \qquad
g_i({y},{t}) = \int_{B_1^{4} \times \sqg{y}} \sum_{j=1}^{n-4}
\brs{\xi_{j} \hook {F}_{{t}}^i }^2 \, dx.
\end{align*}
Using these two quantities, we define two local Hardy-Littlewood maximal
functions of $f_i$ on $B_1^{n-4}$ and $g_i$ on $B_1^{n-4} \times
\brk{-1,-\tfrac{1}{8}}$ by, for $y \in B_1^{n-4}$,
\begin{align*}
M(f_i)(y)
&= \sup_{0 \leq r \leq 1} r^{4-n} \int_{B_r^{n-4}(y)} f_i \, d{y}.
\end{align*}
Furthermore for $\prs{y,t} \in B_1^{n-4} \times \brk{-1,-\tfrac{1}{8}}$,
\begin{align*}
M (g_i) (y, t)&= \sup_{0<r<1} r^{2-n} \int_{{t}-r^2}^{{t}} \int_{B_r^{n-4}(y)}
g_i \, dy \, dt.
\end{align*}
By applying Lemma \ref{wktostronglemma10} we can choose a subsequence such that
$\int_{B_1^{n-4}} f_i dy \leq 4^{-i}$.  Combining this with the Hardy-Littlewood
weak $L^1$ estimate we obtain a subsequence such that
\begin{align*}
\mu \{ y \mid M(f_i) \geq 2^{-i} \} \leq \frac{C}{2^{-i}} \int_{B_1^{n-4}} f_i
\, dy
\leq C 2^{-i}.
\end{align*}
In particular, for $I$ chosen sufficiently large we have
\begin{align*}
\mu \left( \bigcup_{i \geq I} \{ y \mid M(f_i) \geq 2^{-i} \} \right) &\leq\ C
2^{-I} < \frac{1}{2} \mu \prs{B_1^{n-4}}.
\end{align*}
Thus $B_1^{n-4} \backslash \bigcup_{i \geq I} \{ y | M(f_i) \geq 2^{-i} \}$ is
nonempty, and any point $y$ in that set satisfies
\begin{align*}
\lim_{i \to \infty} M(f_i)(y) = 0.
\end{align*}
Combining this with an identical argument yields $(y,t) \in
B_{1/2}^{n-4} \times [-\frac{1}{2},-\frac{1}{4}]$ and a further
subsequence such that
\begin{equation*}
\lim_{i \to \infty} M(g_i)(y,t) = 0, \qquad \lim_{i \to \infty} M(f_i)(y) = 0.
\end{equation*}
The result follows.
\end{proof}
\end{lemma}

For the following Lemma we will suppress bundle indices for notational
simplicity. Moreover, we will refer to coordinate directions $\frac{\del}{\del
x^i}$ with unbarred indices, and $\frac{\del}{\del y^i}$ directions with barred
indices.  For an index which runs over both types of vectors we use $I$ and
$J$. 

\begin{lemma}  \label{wktostronglemma30} One has
\begin{align*}
\frac{\del}{\del y_k}& \brk{\int_{t - \delta_0^2}^{t} \int_{\bR^4}\phi^2(x)
\left[ \brs{ F_{\N }}^2 \right|_{(x,y,t)} \, dx \, dt} \\
&= 4 \int_{t - \delta_0^2}^{t} \int_{\bR^4} \left( (\N_j\phi^2) F_{Ij} -
\phi^2 \left( \frac{\del \N}{\del t} \right)_I \right) F_{\bk I} \, dx \, dt - 4
\frac{\del}{\del y_j}  \brk{\int_{t -
\delta_0^2}^{t} \int_{\bR^4} \phi^2 \prs{ F_{I \bj} F_{\bk I}} \, dx \, dt}.
\end{align*}

\begin{proof} With the notational conventions as described above we have
\begin{align}
\begin{split}\label{eq:Ftriangle}
\N_{\bk} \prs{ F_{ I J }  F_{ I J } }
&= 2 F_{ I J } \prs{ \N_{\bk}F_{ I J }  }\\
&= - 2 F_{ I J } \prs{ \N_{J} F_{\bk I }  + \N_{I} F_{J \bk}} \\
&= - 4 F_{ I J } \prs{ \N_{J} F_{\bk I } } \\
&= - 4 \N_{J} \prs{ F_{ I J }  F_{\bk I } }  - 4 \prs{ \N_{J} F_{ J I }}  F_{\bk
I }.
\end{split}
\end{align}
Lastly, we expand out
\begin{align}
\begin{split}\label{eq:Fcoordmanip}
\brs{F}^2
&= F_{IJ} F_{IJ} = F_{i j} F_{ij }+F_{i \bj} F_{i \bj } + F_{\bi j} F_{\bi j}+
F_{\bi \bj} F_{\bi \bj}.
\end{split}
\end{align}
If we differentiate \eqref{eq:Fcoordmanip} and apply 
\eqref{eq:Ftriangle} to the resulting terms then this breaks down into
\begin{align}
\begin{split}\label{eq:gradF2lem}
\N_{\bk} \brk{ \brs{F}^2 } 
&= -4 \N_j \brk{F_{ij} F_{\bk i}}  - 4 \N_{\bj}\brk{ F_{\bi \bj} F_{\bk \bi} } -
4 \N_{\bj}  \brk{ F_{i \bj} F_{\bk i}} - 4 \N_{j}  \brk{ F_{\bi j} F_{\bk \bi}}
\\
& \hsp - 4 \prs{\N_j F_{ji}} F_{\bk i} - 4 \prs{\N_{\bj} F_{\bj \bi }} F_{\bk
\bi} - 4 \prs{ \N_{\bj} F_{\bj i} } F_{\bk i} - 4 \prs{ \N_{j} F_{j \bi} }
F_{\bk \bi} \\
& =-4 \prs{  \N_j \brk{ F_{ij} F_{\bk i} + F_{\bi j} F_{\bk \bi} } + \N_{\bj}
\brk{F_{\bi \bj} F_{\bk \bi} + F_{i \bj} F_{\bk i}} }  + 4 D_{J}^* F_{JI} F_{\bk
I}.
\end{split}
\end{align}
With these pointwise quantities, we integrate \eqref{eq:gradF2lem} with a cutoff
function $\phi$ and obtain
\begin{align*}
\frac{\del}{\del y_k} & \brk{\int_{t - \delta_0^2}^{t} \int_{\bR^4}\phi^2(x)
\left[
\brs{ F_{\N }}^2 \right|_{(x,y,t)}\, dx \, dt} \\
&=- 4 \int_{t - \delta_0^2}^{t} \int_{\bR^4} \phi^2 \N_j \brk{ F_{ij} F_{\bk
i}  + F_{\bi j} F_{\bk\bi}} \, dx \, dt - 4 \int_{t - \delta_0^2}^{t}
\int_{\bR^4}  \phi^2 \N_{\bj} \brk{ F_{i
\bj} F_{\bk i}  + F_{\bi \bj} F_{\bk \bi}} \, dx \, dt \\
& \hsp + 4 \int_{t - \delta_0^2}^{t} \int_{\bR^4} \phi^2 \prs{D_J^* F_{JI}}
F_{\bk I} \, dx \, dt \\
&= 4 \int_{t - \delta_0^2}^{t} \int_{\bR^4} \prs{\N_j\phi^2} \prs{ F_{ij}
F_{\bk i}  + F_{\bi j} F_{\bk\bi}} \, dx \, dt - 4 \frac{\del}{\del y_j}
\brk{\int_{t -
\delta_0^2}^{t} \int_{\bR^4} \phi^2 \prs{ F_{i \bj} F_{\bk i}  + F_{\bi \bj}
F_{\bk \bi}} \, dx \, dt} \\
& \hsp - 4 \int_{t - \delta_0^2}^{t} \int_{\bR^4} \phi^2 \prs{ \del_t \N_I}
F_{\bk I} \, dx \, dt\\
&= 4 \int_{t - \delta_0^2}^{t} \int_{\bR^4} \left( (\N_j\phi^2) F_{Ij} -
\phi^2 \left(\frac{\del \N}{\del t} \right)_I \right) F_{\bk I} \, dx \, dt - 4
\frac{\del}{\del y_j} \brk{\int_{t -
\delta_0^2}^{t} \int_{\bR^4} \phi^2 \prs{ F_{I \bj} F_{\bk I}} \, dx \, dt},
\end{align*}
as required.
\end{proof}
\end{lemma}

We will use this lemma in conjunction with ``Allard's strong constancy lemma,''
an effective version of the Divergence Theorem which we restate here for
convenience.

\begin{lemma} (\cite{Allard} pp. 3) \label{allardlemma} Suppose $\psi$, $f$, and
$Z$ are smooth on
$B_1$ and satisfy
\begin{equation*}
\N \psi = f + \divg Z.
\end{equation*}
and
\begin{equation*}
\brs{\brs{f}}_{L^1(B_1)} + \brs{\brs{Z}}_{L^1\prs{B_1}} \leq \delta.
\end{equation*}
Then for all $\delta_1 > 0$, there is a $\delta_0 > 0$, depending on $\delta_1$
and $\brs{\brs{\psi}}_{L^1(B_1)}$ such that, whenever $\delta \leq \delta_0$,
\begin{equation*}
\brs{\brs{\psi - \bar{\psi}}}_{L^1(B_1)} \leq \delta_1.
\end{equation*}
where $\bar{\psi}$ denotes the average value of $\psi$ on $B_1$.
\end{lemma}

\begin{lemma} \label{wktostronglemma40} Given a $(y,t) \in B_{1/2}^{n-4} \times
[-\frac{1}{2},-\frac{1}{4}]$ as in Lemma \ref{wktostronglemma20}, there exists a
universal constant $\gL$
and sequences $x_i \to 0, \gd_i \to 0$ such that
\begin{gather*}
\begin{split}
\frac{\ge_0}{\gL} &=\ \gd_i^{-2} \int_{B_{\gd_i}^{4}(x_i) \times [t -
\gd_i^2,t]} \brs{F_{{\N}^i}}^2(x,y,t) \, dx \, dt\\
&=\ \max \left\{ \gd_i^{-2} \int_{B_{\gd_i}^{4} (\til{x}) \times [t -
\gd_i^2,t]} \brs{F_{{\N}^i}}^2(\til{x},y,t) \, dx\, dt \mid \til{x} \in
B_{\frac{1}{2}}^{4} \right\}.
\end{split}
\end{gather*}
\begin{proof} Given $(y,t)$, we fix some $\gL > 0$, then as each blowup
connection
${\N}^i_t$ is smooth we may first choose a constant $\gd_i$ which is the
smallest positive number such that
\begin{align} \label{wktostronglemma4011}
\max \left\{ \gd_i^{-2} \int_{B_{\gd_i}^{4}(\til{x}) \times [t -
\gd_i^2,t]} \brs{{{F}^i}}^2(\til{x},y,t) \ |\ \til{x} \in
B_{\frac{1}{2}}^{4} \right\} = \frac{\ge_0}{\gL}.
\end{align}
Choosing $x_i$ as some point in realizing the maximum defined above, all that
remains to check is that $\gd_i \to 0$, $x_i \to 0$.  First, suppose $\gd_i \geq
\gd_0 > 0$.  Fix $\phi \in C_0^{\infty}(B_{\gd_0}^{4}, [0,1])$, and let
\begin{align*}
\psi(y) := \gd_0^{-2} \int_{t - \gd_0^2}^{t} \int_{B_{\gd_0}^4} \phi^2
\brs{{{F}^i}}^2 (x,y,s) \, dx \, ds.
\end{align*}
Now observe that the result of Lemma \ref{wktostronglemma30} can be interpreted
as $\N \psi = f + \divg Z$, with $f$ and $Z$ defined by the equality.  It
follows from Lemma \ref{wktostronglemma20} that
\begin{align*}
\lim_{i \to \infty} \brs{\brs{f}}_{L^1(B_{\gd_0}^{n-4})} +
\brs{\brs{Z}}_{L^1(B_{\gd_0}^{n-4})} = 0.
\end{align*}
Then we observe using Lemma \ref{allardlemma} and (\ref{wktostronglemma4011})
that
\begin{gather}\label{wktostronglemma4015}
\begin{split}
\lim_{i \to \infty} \gd_0^{2-n} \int_{P_{\gd_0}((0,y),t)} \brs{{{F}^i_t}}^2 dV
dt
&=\ \lim_{i \to \infty} {\bar{\psi}}\\
&=\ \lim_{i \to \infty} \gd_0^{-4} \int_{B_{\gd_0}^{4}} {\bar{\psi}} \, dx\\
&=\ \lim_{i \to \infty} \gd_0^{-4} \int_{B_{\gd_0}^{4}} \prs{\bar{\psi} - \psi +
\psi} \, dx\\
&\leq\ \lim_{i \to \infty} \left[ \gd_0^{-4} \brs{\brs{\psi -
\bar{\psi}}}_{L^1(\bR^4)} + \sup_{B_{\gd_0}^{4}} \psi \right]\\
&\leq\ \frac{\ge_0}{\gL}.
\end{split}
\end{gather}
This contradicts that $((0,y),t) \in \Sigma^*$, hence $\gd_i \to 0$.  Now we
note that the sequence $((x_i,y),t)$ develops concentration of $\brs{F^i_t}$,
and
hence must limit to a singular point, which forces $x_i \to 0$.
\end{proof}
\end{lemma}

With this sequence we can perform a further rescaling to finally obtain a
Yang-Mills connection as blowup limit.  In particular, define the blowup
sequence
\begin{align*}
{\til{\gG}} ^i (x,y,t) &=\ \gd_i \gG ((x_i,y_i) + (\gd_i x, \gd_i y), t_i +
\gd_i^2 t).
\end{align*}
Let us observe some basic properties of this blowup sequence.  In particular, by
rescaling the estimates of Lemmas \ref{wktostronglemma20} and
\ref{wktostronglemma40} we obtain
\begin{gather} \label{wktostronglemma4020}
\begin{split}
\frac{\ge_0}{\gL} &=\ \int_{B_{1}^{4} \times [-1,0]}
\brs{\til{F}^i}^2(0,0,t) \, dx \, dt = \max \left\{  \int_{B_1^{4}
(\til{x}) \times [-1,0]} \brs{\til{F}^i}^2(x,0,t)\, dx \, dt \ |\ x \in
\gd_i^{-1}
B_{\frac{1}{2}}^{4} \right\},\\
0 &=\ \lim_{i \to \infty} \sup_{r \in (0,\frac{1}{4 \gd_i})} r^{4-n}
\int_{B_r^{n-4}(0)} \int_{B_{\frac{1}{2 \gd_i}}^{4} \times \{0\} \times
[-\gd_i^{-2},0]} \brs{\frac{\del \til{\N}^i}{\del t}}^2 \, dx \,  dt \, dy,\\
0 &=\ \lim_{i \to \infty} \sup_{r \in (0,\frac{1}{4 \gd_i})} r^{2-n}
\int_{B_r^{n-4}(0) \times
[- r^2,0]} \int_{B_{\frac{1}{2 \gd_i}}^{4}  \times \{y\}} \sum_{j=1}^{n-4}
\brs{\frac{\del}{\del y_j} \hook \til{F}^i_t}^2 \, dx \, dt \, dy.
\end{split}
\end{gather}

\begin{lemma} \label{wktostronglemma50} The sequence $\{\til{\N}^i_t\}$
converges
strongly to a nonflat Yang-Mills connection on $S^4$.
\begin{proof} We use the estimates of (\ref{wktostronglemma4020}) and argue as
in the estimate (\ref{wktostronglemma4015}) to show an
energy estimate of the form
\begin{align} \label{wktostronglemma5010}
\int_{P_{\frac{3}{2}} ((\til{x},0),0)} \brs{{\til{F}^i_t}}^2 \, dV \, dt \leq
\frac{\ge_0}{2} \quad \mbox{for all } \til{x} \in \gd_i^{-1}
\prs{B_{\frac{1}{2}}^{4}}.
\end{align}
Given this, we can complete the proof as follows.  There is a local $H^{1,2}$
estimate for $\til{\N}^i_t$ and hence we can choose a subsequence so that
$\til{\N}^i_t \to \til{\N}^{\infty}_t$ weakly in $H_{\mbox{loc}}^{1,2}(\mathbb
R^n
\times \mathbb (- \infty, 0])$.  However, using (\ref{wktostronglemma4020}) we
have that
\begin{align*}
\int_{\mathbb R^n \times (-\infty,0]} \prs{ \brs{ \frac{\del 
\til{\N}^{\infty}_t}{\del t}}^2 +
\sum_{j=1}^{n-4} \brs{ \frac{\del}{\del y_j} \hook \til{F}^{\infty}_t}^2
} \, dV \, dt= 0.
\end{align*}
Using (\ref{wktostronglemma5010}) and Theorem \ref{thm:eregularity} we obtain
convergence of $\til{\N}^i_t$ to $\til{\N}^{\infty}_t$ in $C^{k,\ga}(K)$ for any
compact set $K \subset \mathbb R^n \times (-\infty, 0]$.  In particular, using
\eqref{wktostronglemma4020} we obtain
\begin{align*}
\frac{\ge_0}{\gL} \leq \int_{\mathbb R^4} \brs{\til{F}^{\infty}_t}^2 \, dx <
\infty,
\end{align*}
hence $\til{\N}^{\infty}_t$ is not flat. The result follows.
\end{proof}
\end{lemma}
\noindent Lemma \ref{wktostronglemma50} finishes the proof of the theorem.
\end{proof}
\end{thm}

\begin{proof} [Proof of Corollary \ref{cor:weaktostrong}] Without loss of
generality by an overall rescaling we assume $T \geq 2$.  Choose any sequence
$\{t_i\} \to T$, and observe that the sequence of solutions given by restricting
the given solution to $[t_i - 1, t_i]$ satisfies the hypotheses of Theorem
\ref{fullweakcompactthm}.  By hypothesis that $T$ is maximal we know that
$\Sigma \neq \varnothing$.  As shown in Theorem \ref{fullweakcompactthm} the
point
$z \in \Sigma$ is a point of entropy concentration.  Thus we can choose a
sequence of radii $r_i \to 0$ and rescale the parabolic balls $P_{r_i}(z_0)$ to
unit
size, to obtain a sequence of solutions with finite, nonzero entropy.  It
follows easily that the hypotheses of Theorem \ref{fullweakcompactthm} hold for
this sequence.  If the sequence does not converge strongly in $H^{1,2}$, Theorem
\ref{thm:weaktostrong} yields the further blowup sequence which converges to a
Yang-Mills connection on $S^4$.  If this sequence does converge strongly in
$H^{1,2}$, as the $\Psi$ functional is becoming constant along the blowup
sequence, the 
second term of the entropy monotonicity formula of (\ref{Psimono}) converges to
zero, which implies that the
blowup limit is a soliton.
\end{proof}


%
%
\end{document}